\documentclass[12pt]{article}
 
 \usepackage{amsmath,amssymb, amsthm}
 
\usepackage{amsmath,amssymb,amsfonts,amsthm}

\usepackage{virginialake}

\usepackage{adjustbox} 
\usepackage{float}
\usepackage{tikz}
\usetikzlibrary{positioning,arrows}
\usetikzlibrary{decorations.pathreplacing}
\tikzset{
modal/.style={>=stealth,shorten >=1pt,shorten <=1pt,auto,node distance=1.5cm,semithick},
world/.style={circle,draw,minimum size=0.5cm,fill=gray!15},
point/.style={circle,draw,inner sep=0.5mm,fill=black},
reflexive above/.style={->,loop,looseness=7,in=120,out=60},
reflexive below/.style={->,loop,looseness=7,in=240,out=300},
reflexive left/.style={->,loop,looseness=7,in=150,out=210},
reflexive right/.style={->,loop,looseness=7,in=30,out=330}}

\usepackage{mdframed}

\usepackage{doc}

\usepackage{xcolor}
\usepackage{xspace}
\usepackage{adjustbox}

\usepackage{amsfonts}
\usepackage{graphicx}
\newtheorem{thm}{Theorem}[section]
\newtheorem{prop}[thm]{Proposition}
\newtheorem{cor}[thm]{Corollary}

\newtheorem{lem}[thm]{Lemma}
\newtheorem{dfn}[thm]{Definition}

\newcommand{\M}{\mbox{$\mathfrak{M}$}}
\newcommand{\fl}{\mbox{$\mathfrak{F}$}}
\newcommand{\PP}{\mbox{$\mathcal{P}$}}

\newcommand{\AND}{\wedge}
\newcommand{\OR}{\vee}
\newcommand{\IMP}{\rightarrow}
\newcommand{\IFF}{\leftrightarrow}

\newcommand{\logicname}[1]{{\sf#1}\xspace}
\newcommand{\f}{\logicname{F}}
\newcommand{\sfour}{\logicname{S4}}
\newcommand{\sj}{\logicname{SJ}}
\newcommand{\bpc}{\logicname{BPC}}
\newcommand{\ipl}{\logicname{IPL}}
\newcommand{\klogic}{\logicname{K}}
\newcommand{\wf}{\logicname{WF}}

\newcommand{\gwf}{\logicname{GWF}}
\newcommand{\gwfc}{\logicname{GWFC}}

\newcommand{\gwfi}{\logicname{GWFI}}

\newcommand{\propset}{\mathsf{Prop}}

\begin{document}

\title{Sequent Calculi for Some Subintuitionistic Logics }

\author{\textbf{Fatemeh Shirmohammadzadeh Maleki}\\
Department of Logic, Iranian Institute of Philosophy\\
 Arakelian 4, Vali-e-Asr, Tehran, Iran,~f.shmaleki2012@yahoo.com}

\maketitle

\begin{abstract}

This paper investigates sequent calculi for certain weak subintuitionistic logics. We establish that weakening and contraction are height-preserving admissible for each of these calculi, and we provide a syntactic proof for the admissibility of the cut rule for most of these sequent calculi. We also demonstrate the equivalence of these calculi to their corresponding axiomatic systems, thereby confirming their soundness and completeness with respect to neighbourhood semantics.

\end{abstract}


\section{Introduction}

Subintuitionistic logics as a theme were first studied by Corsi~\cite{a4}, who introduced a basic system \f in a Hilbert-style proof system. The system \f
has Kripke frames in which the assumption of preservation of truth is dropped and which are not assumed to be reflexive or transitive. Corsi also introduced G\"odel-type translations of these systems into modal logic and showed that \f can be translated into the modal logic \klogic just as  \ipl into \sfour. 
Restall~\cite{a2} defined a similar system \sj (see also~\cite{Dic}). A much studied extension of \f was introduced by Visser ~\cite{a3}. He introduced Basic logic \bpc in natural deduction form and proved completeness of \bpc for finite irreflexive Kripke models. Many results on this logic were obtained by Ardeshir and Ruitenburg (see e.g.~\cite{ Ar1, Ar}). 

Building on this line of research, de Jongh and Shirmohammadzadeh Maleki introduced weaker subintuitionistic logics based on neighborhood semantics in~\cite{Dic3, Dic4, FD}, departing from the traditional Kripkean frames. Specifically, they focus on the system \wf, which is significantly weaker than {\sf F}, along with its extensions. These systems have modal companions, establishing a correspondence between subintuitionistic and modal logics~\cite{Dic4, FD6, Dic5}. Sequent systems for the modal companions of these weak subintuitionistic logics have been provided in~\cite{Ra, Eu}.

Kikuchi \cite{kik} introduced a dual-context style sequent calculus for \f. Ishigaki and Kashima \cite{ish} introduced a sequent calculus for \f and proved the cut-elimination theorem semantically. Yamasaki and Sano \cite{yam} provided a labelled sequent calculi for \f and showed that their calculi admit the cut rule. Aboolian and Alizadeh \cite{ab} introduced a cut-free G3-style sequent calculus for \f in order to prove Lyndon's and Craig's interpolation properties for this logic.
Sasaki \cite{sa}, Aghaei and Ardeshir \cite{ag} introduced two types of sequent calculus for \bpc and proved Craig's interpolation property sintactically for this logic \cite{ag}. Tesi \cite{Tesi}
introduced nested calculi for subintuitionistic logics corresponding
(in the sense of the modal interpretation) to the modal logics {\sf K}, {\sf T}, {\sf K4} and for Visser’s logic \bpc.

In this paper, we introduce sequent systems for several weak subintuitionistic logics, including \wf, \f and logics lying between \wf and \f. We explore the admissibility of structural rules such as weakening, contraction, and cut, and provide proofs of the equivalence between Hilbert-style systems and their corresponding sequent systems. Additionally, a new cut-free system {\sf GF} is presented, shown to be equivalent to \f.

The paper is structured as follows. Section \ref{sec:modneigh} summarizes the basic notions of axiomatic systems and neighbourhood semantics for some subintuitionistic logics. Section \ref{modneigh1} presents the sequent calculus system for the basic subintuitionistic logic \wf, and shows that weakening and contraction are height-preserving admissible, and that cut is syntactically admissible. In Section \ref{some}, we present sequent calculus systems for some other subintuitionistic logics between \wf and \f, and then prove the admissibility of structural rules for some of them. Moreover, we establish the equivalence between the Hilbert-style proof systems and the corresponding sequent calculus systems. Finally, in Section \ref{fcdi}, a new cut-free sequent calculus {\sf GF} is introduced and we showed that {\sf GF} and \f are equivalent.

\section{Subintuitionistic logics}\label{sec:modneigh}

\subsection{Syntax}
In this subsection we recall the Hilbert-style system for the basic subintuitionistic logic {\sf WF} and some extensions of it. The results here have been proved before in \cite{FD, Dic4}. 

The language of the subintuitionistic logics
is generated from a countable set  $\propset$ of atomic propositions, which we denote by lowercase letters $p, q, \dots$, through the   
connectives $ \vee, \wedge, \rightarrow $ and the propositional constant $ \bot $. 
We denote formulas of the language with uppercase latin letters $A, B, C, \ldots$. 
We define the symbol $ \leftrightarrow  $ by $ A\leftrightarrow B \equiv (A\rightarrow B)\wedge (B\rightarrow A) $ as abbreviation. 

\begin{figure}[t]
	\begin{center}
		
	\begin{tabular}{c l @{\hspace{1cm}} cl }
		1. & $ A \IMP (A \OR B )$ & 8. & $A \IMP A$ \\
		2. & $ B \IMP (A \OR B )$ & 9. &  $\vliinf{}{}{A\IMP C}{A \IMP B}{B \IMP C}$\\
		3. & $ (A \AND B) \IMP A$ & 10. & $\vliinf{}{}{A \IMP (B \AND C)}{A \IMP B }{A \IMP C}$\\
		4. & $ (A \AND B) \IMP B$ & 11.  & $\vliinf{}{}{(A \OR B ) \IMP C}{A \IMP C}{B \IMP C}$\\
		5. & $\vliinf{}{}{B}{A}{A \IMP B}$ &12. & $\vliinf{}{}{ A \AND B}{A}{B}$\\
		6. & $\vlinf{}{}{B \IMP A}{A} $& 13. & $\vliinf{}{}{(A \IMP C) \IFF (B \IMP D) }{A \IFF B}{C \IFF D}$\\
		7. & $ A \AND(B \OR C) \IMP (A \AND B ) \OR (A \AND C) $ &14. & $\bot \IMP A$  \\
	\end{tabular}
\end{center}
\caption{The Hilbert-style system for $\wf$ }
\label{fig:axioms:wf}
\end{figure}

 \begin{dfn}\label{def:hilbert:wf} 
The Hilbert-style axiomatization of the basic subintuitionistic logic \wf consists of the axioms and inference rules reported in Figure~\ref{fig:axioms:wf}. 
\end{dfn}

In Figure~\ref{fig:axioms:wf}, the rules are to be applied in such a way that, if the formulas above the line are theorems of \wf, then the formula below the line is a theorem as well and we will call rules 5, 6 and 12 the  modus ponens (MP), a fortiori (AF) and conjunction rules, respectively.
The basic notion $ \vdash_{\wf} A$ means that $ A $ can be drived from the axioms of the \wf by means of its rules. But, when one axiomatizes local validity, not all rules in a Hilbert type system have the same status when one considers deductions from assumptions.
In modal logic the rule of necessitation can only be applied to prove theorems, not to derive conclusions from assumptions.
Almost all the rules of the \wf system are subject to similar limitations.
In the case of deduction from assumptions, we restrict all the rules except the conjunction rule, the restriction on modus ponens is slightly weaker than on the other rules; when concluding $ B $ from $ A $, $ A\rightarrow B $ only the implication $ A\rightarrow B $ need be a theorem. These considerations result in the following definition. 

In this section we will denote that a formula $ A $ is derivable from $ \Gamma $ in the Hilbert style system of {\sf WF} as $ \Gamma\,{\vdash_{{\sf WF}}} A $.

 \begin{dfn}\label{hh}
We define $ \Gamma\,{\vdash_{\wf}} A $ iff there is a  derivation of A from $ \Gamma $ using the rules 6, 9, 10 , 11 and 13 of Figure~\ref{fig:axioms:wf}, only when there are  no assumptions, and the rule 5, MP, only when the derivation of $ A\rightarrow B $ contains no assumptions.
\end{dfn}  
For example if we assume that $ \Gamma= \wf \cup \lbrace p\rbrace $,  then $ \Gamma \vdash_{\wf} p $ and $ \Gamma\nvdash_{\wf} q\rightarrow p$.

By executing the definition of $ \Gamma\,{\vdash_\wf} A $, a weak form of the deduction theorem with a single assumption is obtained:

\begin{thm}\label{z3}
{\rm (Weak Deduction theorem,~\cite{FD} Theorem 2.19)}
\begin{enumerate}
\item[$ \bullet $] $A\vdash_{{\sf WF}} B  ~$ iff $ ~\vdash_{{\sf WF}} A\rightarrow B $.
\item[$ \bullet $] $A_1,\dots,A_n\vdash_{{\sf WF}} B  ~$ iff $~ \vdash_{{\sf WF}} A_1\wedge\dots\wedge A_n\rightarrow B $.
\end{enumerate}
\end{thm}

As usual, we can extend \wf by adding various rules and axiom schemes.
In this paper, we will consider the following axioms and rules:
$$ \frac{A\rightarrow B\vee C~~~~~C\rightarrow A\vee D~~~~~A\wedge  D\rightarrow B~~~~~ C\wedge B\rightarrow D}{(A\rightarrow B)\leftrightarrow (C\rightarrow D)}~~~~{\sf N}$$
$$~~~~~ \frac{C\rightarrow A\vee D~~~~~~~ C\wedge B\rightarrow D}{(A\rightarrow B)\rightarrow (C\rightarrow D)}~~~~~~~~~~{\sf N_{2}}$$
$$(A\rightarrow B)\wedge (B\rightarrow C)\rightarrow(A\rightarrow  C)~~~~~~~~~~{\sf I}$$
$$(A\rightarrow B)\wedge (A\rightarrow C)\rightarrow(A\rightarrow B\wedge C)  ~~~~{\sf C}$$
$$(A\rightarrow C)\wedge (B\rightarrow C)\rightarrow(A\vee B \rightarrow C)  ~~~~{\sf D}$$
$$(A\rightarrow B\wedge C)\rightarrow (A\rightarrow B)\wedge (A\rightarrow C) ~~~~{\sf \widehat{C}} $$
$$ (A\vee B \rightarrow C)\rightarrow (A\rightarrow C)\wedge (B\rightarrow C) ~~~~{\sf \widehat{D}}$$

If $ \Gamma\subseteq\lbrace {\sf I}, {\sf C},{\sf D}, {\sf N},{\sf \widehat{C}}, {\sf \widehat{D}},   {\sf N_{2}}\rbrace $, we will write ${\sf WF\Gamma }$ for the logic obtaines from {\sf WF} by adding to  {\sf WF}  the schemas and rules in $ \Gamma $ as new axioms and rules.
Instead of ${\sf WFN}$  and ${\sf WFN_{2}}$ we will write ${\sf WF_N}$  and ${\sf WF_{N_{2}}}$, respectively.

The logic \f is the smallest set of formulas closed under instances of \wf, {\sf C}, {\sf D} and {\sf I} \cite{FD}. Moreover, it is shown that the logic $ {\sf WF_{N_{2}}C} $ equals the logic \f \cite{Dic4}. 

In figures \ref{fig:wfc}, \ref{fig:wfd} and \ref{fig:wfi} the reader finds the lattice of subintuitionistic logics between \wf and \f. 
The relations between the logics in distinct ones of the three cubes is mostly yet unclear and within the same cube all the relations are strict.

In the subsequent sections,  we will present the first sequent calculus for the systems \wf,  ${\sf WF_N}$, ${\sf WF_{N_{2}}}$, ${\sf WFI}$, ${\sf WFC}$, ${\sf WFD}$, ${\sf WF}{\sf \widehat{C}}$, ${\sf WF}{\sf \widehat{D}}$, ${\sf WFCI}$, ${\sf WFDI}$ and alternative sequent calculus system for \f. We will demonstrate the admissibility of cut for the sequent calculus systems of \wf, ${\sf WF_N}$, ${\sf WF_{N_{2}}}$, ${\sf WFI}$, ${\sf WFCI}$, ${\sf WFDI}$ and \f. However, we have not yet been able to establish cut admissibility for the sequent calculus systems of ${\sf WFC}$, ${\sf WFD}$, ${\sf WF}{\sf \widehat{C}}$ and ${\sf WF}{\sf \widehat{D}}$.

\begin{figure}[t]
\begin{center}
$${\begin{tikzpicture}[line cap=round,line join=round,>=triangle 45,x=1.25cm,y=1.10cm]
\draw (-6.72,3.46) node[anchor=north west] {${\sf  WF\widehat{\sf C}C}$};
\draw (-3.72,4) node[anchor=north west] {${\sf WF_{N_{2}}C}$};
\draw (-6.54,1.4) node[anchor=north west] {${\sf  WF  \widehat{\sf C}}$};
\draw (-4.34,0.5) node[anchor=north west] {{\sf WF}};
\draw (-4.44,2.58) node[anchor=north west] {${\sf  WF C}$};
\draw (-1.42,1.46) node[anchor=north west] {${\sf  WF_{N}}$};
\draw (-1.52,3.4) node[anchor=north west] {${\sf  WF_{N}C}$};
\draw (-3.5,2.1) node[anchor=north west] {${\sf WF_{N_{2}}}$};
\draw [line width=1pt] (-6,2.84)-- (-5.96,1.32);
\draw [line width=1pt] (-5.74,0.86)-- (-4.3,0.26);
\draw [line width=1pt] (-4,2)-- (-3.96,0.5);
\draw [line width=1pt] (-5.56,2.8)-- (-4.4,2.36);
\draw [line width=1pt] (-3.56,2.3)-- (-1.42,2.8);
\draw [line width=1pt] (-5.5,3.14)-- (-3.8,3.5);
\draw [line width=1pt] (-2.56,3.36)-- (-1.64,3.12);
\draw [line width=1pt] (-1.04,2.66)-- (-1.02,1.48);
\draw [line width=1pt] (-3.58,0.28)-- (-1.36,0.92);
\draw [line width=1pt,dash pattern=on 2pt off 3pt] (-5.56,1.18)-- (-3.48,1.76);
\draw [line width=1pt,dash pattern=on 2pt off 3pt] (-2.56,1.62)-- (-1.56,1.26);
\draw [line width=1pt,dash pattern=on 2pt off 3pt] (-3.02,3.42)-- (-3,2);
\end{tikzpicture}}$$
\end{center}
\caption{Lattice of some subintuitionistic logics  }
\label{fig:wfc}
\end{figure}

\begin{figure}[t]
\begin{center}
$${\begin{tikzpicture}[line cap=round,line join=round,>=triangle 45,x=1.25cm,y=1.10cm]
\draw (-6.72,3.46) node[anchor=north west] {${\sf  WF\widehat{\sf D}D}$};
\draw (-3.72,4) node[anchor=north west] {${\sf WF_{N_{2}}D}$};
\draw (-6.54,1.4) node[anchor=north west] {${\sf  WF  \widehat{\sf D}}$};
\draw (-4.34,0.5) node[anchor=north west] {{\sf WF}};
\draw (-4.44,2.58) node[anchor=north west] {${\sf  WF D}$};
\draw (-1.42,1.46) node[anchor=north west] {${\sf  WF_{N}}$};
\draw (-1.52,3.4) node[anchor=north west] {${\sf  WF_{N}D}$};
\draw (-3.5,2.1) node[anchor=north west] {${\sf WF_{N_{2}}}$};
\draw [line width=1pt] (-6,2.84)-- (-5.96,1.32);
\draw [line width=1pt] (-5.74,0.86)-- (-4.3,0.26);
\draw [line width=1pt] (-4,2)-- (-3.96,0.5);
\draw [line width=1pt] (-5.56,2.8)-- (-4.4,2.36);
\draw [line width=1pt] (-3.56,2.3)-- (-1.42,2.8);
\draw [line width=1pt] (-5.5,3.14)-- (-3.8,3.5);
\draw [line width=1pt] (-2.56,3.36)-- (-1.64,3.12);
\draw [line width=1pt] (-1.04,2.66)-- (-1.02,1.48);
\draw [line width=1pt] (-3.58,0.28)-- (-1.36,0.92);
\draw [line width=1pt,dash pattern=on 2pt off 3pt] (-5.56,1.18)-- (-3.48,1.76);
\draw [line width=1pt,dash pattern=on 2pt off 3pt] (-2.56,1.62)-- (-1.56,1.26);
\draw [line width=1pt,dash pattern=on 2pt off 3pt] (-3.02,3.42)-- (-3,2);
\end{tikzpicture}}$$
\end{center}
\caption{Lattice of some subintuitionistic logics }
\label{fig:wfd}
\end{figure}

\begin{figure}[t]
\begin{center}
$$\begin{tikzpicture}[line cap=round,line join=round,>=triangle 45,x=1.25cm,y=1.10cm]
\draw (-6.72,3.46) node[anchor=north west] {${\sf  WFCI }$};
\draw (-3.72,4) node[anchor=north west] {${\sf WFCDI=F}$};
\draw (-6.54,1.4) node[anchor=north west] {${\sf  WFI}$};
\draw (-4.34,0.5) node[anchor=north west] {{\sf WF}};
\draw (-4.44,2.58) node[anchor=north west] {${\sf  WFC}$};
\draw (-1.42,1.46) node[anchor=north west] {${\sf  WFD}$};
\draw (-1.52,3.4) node[anchor=north west] {${\sf  WFCD}$};
\draw (-3.5,2.1) node[anchor=north west] {${\sf WFDI}$};
\draw [line width=1pt] (-6,2.84)-- (-5.96,1.32);
\draw [line width=1pt] (-5.74,0.86)-- (-4.3,0.26);
\draw [line width=1pt] (-4,2)-- (-3.96,0.5);
\draw [line width=1pt] (-5.56,2.8)-- (-4.4,2.36);
\draw [line width=1pt] (-3.56,2.3)-- (-1.42,2.8);
\draw [line width=1pt] (-5.5,3.14)-- (-3.8,3.5);
\draw [line width=1pt] (-2.56,3.36)-- (-1.64,3.12);
\draw [line width=1pt] (-1.04,2.66)-- (-1.02,1.48);
\draw [line width=1pt] (-3.58,0.28)-- (-1.36,0.92);
\draw [line width=1pt,dash pattern=on 2pt off 3pt] (-5.56,1.18)-- (-3.48,1.76);
\draw [line width=1pt,dash pattern=on 2pt off 3pt] (-2.56,1.62)-- (-1.56,1.26);
\draw [line width=1pt,dash pattern=on 2pt off 3pt] (-3.02,3.42)-- (-3,2);
\end{tikzpicture}$$
\end{center}
\caption{Lattice of some subintuitionistic logics }
\label{fig:wfi}
\end{figure}

\subsection{Semantics}

In the following we first recall the NB-neighbourhood  frames introduced in \cite{FD}, and further studied in \cite{Dic3, Dic4}.
\begin{dfn}
An  \textbf{NB-neighbourhood Frame}  $  \fl\,{=}\,\langle W, N\!B\rangle $  for subintuitionistic logic consists of a non-empty set $W$, and a function
$N\!B$  from $W$ into $ \mathcal{P}((\mathcal{P}(W))^{2} )$ such that:
$$\forall w \in W , ~ \forall X, Y \in \PP(W) ~(X\subseteq Y~\Rightarrow~(X, Y) \in N\!B(w)).$$
\noindent In an \textbf{NB-neighbourhood Model} $\M= \langle W, N\!B, V\rangle $,  $ V\! : \propset\rightarrow \mathcal{P}(W)$ is a valuation function on the set of propositional variables.
 \end{dfn}  

\begin{dfn}
\label{truth}
Let $ \M=\langle W, N\!B, V \rangle $ be an \emph{NB}-neighbourhood model.
\noindent \textbf{Truth} of a propositional formula in a world $w$ is defined inductively as follows.
\begin{enumerate}
\item $ \M,w \Vdash p~~~~~~ ~~\Leftrightarrow~~ w \in V(p)$;
\item $ \M,w \Vdash A\wedge B~~\Leftrightarrow~~ \M,w \Vdash A ~{\rm and}~ \M,w \Vdash B$;
\item $ \M,w \Vdash A\vee B~~\Leftrightarrow ~~\M,w \Vdash A ~{\rm or}~ \M,w \Vdash B$;
\item $ \M,w \Vdash A\rightarrow B~\Leftrightarrow ~\left(  A^{{\mathfrak M}}, B^{{\mathfrak M}}\right)   \in N\!B(w)$;
\item $ \M,w \nVdash \perp,$
\end{enumerate}
where $A^{{\mathfrak M}}= \lbrace w \in W ~|~\M , w  \Vdash A \rbrace$ is the truth set of $ A $.
\end{dfn}

\begin{thm}{\rm (Completeness theorem,~\cite{FD})}\label{4}
 The logic {\sf WF} is sound and strongly complete with respect to the class of \emph{NB}-neighbourhood frames.
 \end{thm}

In order to give soundness and completeness results for extensions of \wf with respect to (classes of) neighbourhood models, we introduce the following definition.

\begin{dfn}
For every neighbourhood frame  $  \fl\,{=}\,\langle W, N\!B\rangle $, we list some relevant properties as follows (here $ X, Y, Z \in \PP(W) $, $ w \in W $ and we denote the complement of $ X $ by $\overline{X} $):
\begin{enumerate}
\item[$ \bullet $] \fl \, is closed under \textbf{intersection} if and only if for all $ w \in W $, if $ (X, Y)\in N\!B(w) $, $ (X, Z)\in N\!B(w) $ then $ (X,  Y\cap Z) \in N\!B(w) $.
\item[$ \bullet $] \fl \, is closed under \textbf{union} if and only if for all $ w \in W $, if $ (X, Y)\in N\!B(w) $, $ (Z, Y)\in N\!B(w) $ then $ (X\cup Z,  Y) \in N\!B(w) $.
\item[$ \bullet $] \fl \, satisfies \textbf{transitivity} if and only if for all $ w \in W $, if $ (X, Y)\in N\!B(w) $, $ (Y, Z)\in N\!B(w) $ then $ (X,  Z) \in N\!B(w) $.
\item[$ \bullet $] \fl \, is closed under \textbf{upset} if and only if for all $ w \in W $, if $ (X, Y)\in N\!B(w) $ and $ Y\subseteq Z $ then $ (X, Z) \in N\!B(w) $.
\item[$ \bullet $] \fl \, is closed under \textbf{downset} if and only if for all $ w \in W $, if $ (X, Y)\in N\!B(w) $ and $ Z\subseteq X $ then $ (Z, Y) \in N\!B(w) $.
\item[$ \bullet $] \fl \, is closed under \textbf{equivalence} if and only if for all $ w \in W $, if $ (X, Y)\in N\!B(w) $ and $ \overline{X} \cup Y =\overline{X^{'}} \cup Y^{'}$ then $ (X^{'}, Y^{'}) \in N\!B(w) $.
\item[$ \bullet $] \fl \, is closed under \textbf{superset equivalence} if and only if for all $ w \in W $, if $ (X, Y)\in N\!B(w) $ and $ \overline{X} \cup Y \subseteq \overline{X^{'}} \cup Y^{'}$ then $ (X^{'}, Y^{'}) \in N\!B(w) $.
\end{enumerate}
\end{dfn}

To see the proof of the following proposition and theorem, you can refer to \cite{Dic3, Dic4, FD}.

\begin{prop}
We have the following correspondence results between formulas and the
properties of the neighbourhood function defined above:
\begin{enumerate}
\item[$ \bullet $] Axiom {\sf C} corresponds to closure under intersection;
\item[$ \bullet $] Axiom {\sf D} corresponds to closure under union;
\item[$ \bullet $] Axiom {\sf I} corresponds to transitivity;
\item[$ \bullet $] Axiom ${\sf \widehat{C}}$ corresponds to closure under upst;
\item[$ \bullet $] Axiom ${\sf \widehat{D}}$ corresponds to closure under downset;
\item[$ \bullet $] Rule {\sf N} corresponds to closure under equivalence;
\item[$ \bullet $] Rule ${\sf N_{2}}$ corresponds to closure under superset equivalence.
\end{enumerate}
\end{prop}

\begin{thm}\label{4b}
If $ \Gamma\subseteq\lbrace {\sf I}, {\sf C}, {\sf D} , {\sf \widehat{C}}, {\sf \widehat{D}}, {\sf N}, {\sf N_{2}}\rbrace $, then  ${\sf WF\Gamma }$ is sound and strongly complete
with respect to the class of \emph{NB}-neighbourhood frames with all the
properties defined by the axioms and rules in $ \Gamma $.
\end{thm}

\section{ The system {\sf  GWF}}\label{modneigh1}

In this section, first we present a system of sequent calculus for subintuitionistic logic \wf with the remarkable property that all structural rules, weakening, contraction, and cut are addmissible in it. 
Then we establish the equivalence between the sequent calculus system and the corresponding Hilbert-style system for \wf.
Sequents are of the form $ \Gamma\Rightarrow \Delta $, where $ \Gamma $ and $ \Delta $ are  finite, possibly empty, multisets of formulas. 
Typically, a sequent $ \Gamma \Rightarrow \Delta $ is interpreted as follows: the conjunction of the formulas in $ \Gamma $ implies the disjunction of the formulas in $ \Delta $.

\begin{dfn}
The rules of the calculus \gwf for the basic subintuitionistic logic \wf are the following ($ p $ atomic):

\vspace{0.5cm}

$$ \dfrac{•}{p, \Gamma\Rightarrow \Delta, p}Ax~~~~~~~~~~~~~~~~~~~~~~ \dfrac{•}{\bot, \Gamma\Rightarrow \Delta}\bot_{L}$$

\vspace{0.3cm}
$$\dfrac{A, B, \Gamma\Rightarrow \Delta}{A\wedge B, \Gamma\Rightarrow \Delta}\wedge_{L}~~~~~~~~~~~~~\dfrac{\Gamma\Rightarrow \Delta, A~~~\Gamma\Rightarrow \Delta, B}{ \Gamma\Rightarrow \Delta, A\wedge B}\wedge_{R}$$

\vspace{0.3cm}
$$\dfrac{A, \Gamma\Rightarrow \Delta~~~B, \Gamma\Rightarrow \Delta}{A\vee B, \Gamma\Rightarrow \Delta}\vee_{L}~~~~~~~~~~~~~\dfrac{ \Gamma\Rightarrow \Delta, A, B}{ \Gamma\Rightarrow \Delta, A\vee B}\vee_{R}~~~$$

\vspace{0.3cm}
$$ \dfrac{A\Rightarrow B}{\Gamma\Rightarrow A\rightarrow B, \Delta}\rightarrow_{R}~~~~~~~~~~\frac{A\Rightarrow B~~~B\Rightarrow A~~~C\Rightarrow D~~~D\Rightarrow C}{\Gamma, A\rightarrow C\Rightarrow B\rightarrow D, \Delta}\rightarrow_{LR}$$

\end{dfn}

As measures for inductive proofs we use the weight of a formula and the height of a derivation. So, in the following we define weight of formulas and height of derivations.

\begin{dfn}
The \textbf{weight}  of a formula $ A $ is defined inductively as follows:

$ w(\bot)=0 $, 

$ w(p)=1~~~ $ for atoms $ p $,

$ w(A \circ B)= w(A)+w(B)+1~~~$ for conjunction, disjunction and implication.
\end{dfn}

\begin{dfn}
A derivation in \gwf is either an axiom, an instance of $ \bot_{L} $,
or an application of a logical rule to derivations concluding its premises. The
\textbf{height} of a derivation is the greatest number of successive applications of rules
in it, where an axiom and $ \bot_{L} $ have height 0.
\end{dfn}
A rule of inference is said to be \textbf{(height-preserving) admissible } in \gwf if, whenever its premises are derivable in \gwf, then also its conclusion is derivable (with at most the same derivation height) in \gwf.

Note that in the rules, the multisets $ \Gamma $ and $ \Delta $ are called \textbf{contexts}, the formula in the conclusion is called \textbf{principal}, the formulas in the premises are called \textbf{active}.

\subsection{Admissibility of structural rules}
In this section we will prove the admissibility of structural rules for the calculus  \gwf.

\begin{lem}\label{1}
The sequent $ C, \Gamma\Rightarrow \Delta, C $ is derivable in ${\sf  GWF }$  for an arbitrary formula $ C $ and arbitrary contexts $ \Gamma $ and $\Delta$.
\end{lem}
\begin{proof}
The proof is by induction on weight of $ C $. If $ w(C)\leq 1 $, then 
$ C $
 is either $ \bot $ or $ C=p $ for some atom 
$ p $, or $ C=\bot\rightarrow\bot$. In the first case, $ C, \Gamma\Rightarrow \Delta, C $ is an instance of $\bot_{L}$; in the second, it is an axiom. If $ C=\bot\rightarrow\bot$, then $ C, \Gamma\Rightarrow \Delta, C $, is derived by:
$$\frac{\frac{}{\bot\Rightarrow\bot}\bot_{L}}{\bot\rightarrow\bot, \Gamma \Rightarrow \Delta, \bot\rightarrow\bot}\rightarrow_{R}$$
The inductive hypothesis assumes that $ C, \Gamma\Rightarrow \Delta, C $ is derivable for all formulas $ C $ with $ w(C)\leq n $, and we need to demonstrate that 
$ D, \Gamma\Rightarrow \Delta, D $  is derivable for formulas $ D $ of weight $ \leq n+1 $. There are three cases to consider. We only consider the case where  $ D=A\rightarrow B $. By the definition of weight, $ w(A)\leq n $ and $ w(B)\leq n $, and we have the derivation
$$\frac{A\Rightarrow A~~~~A\Rightarrow A~~~~B\Rightarrow B~~~~B\Rightarrow B}{A\rightarrow B, \Gamma \Rightarrow A\rightarrow B , \Delta }\rightarrow_{LR}$$
Noting that the contexts are arbitrary and here $ A\Rightarrow A $ and $B \Rightarrow B$ are derivable by the inductive hypothesis (IH).

The proof for the cases $ D=A\wedge B $ and  $ D=A\vee B $ also holds by the inductive hypothesis.
\end{proof}

\begin{thm}\label{1b}
\textbf{Height-preserving weakening.} 
The left and right rules of weakening are height-preserving admissible in \gwf.
$$ \frac{\Gamma\Rightarrow \Delta}{D, \Gamma\Rightarrow \Delta}L_{w}~~~~~~~~~~~~~\frac{\Gamma\Rightarrow \Delta}{\Gamma\Rightarrow \Delta, D}R_{w}$$
\end{thm}

\begin{proof}
The proof is by induction on height of derivation.
If the last rule applied is $ \wedge_{L} $, $ \wedge_{R}$, $ \vee_{L} $ or $ \vee_{R}$, then we have to apply the same rule to the weakened premiss(es), which are
derivable by IH.
If it is by $ \rightarrow_{R} $ or $\rightarrow_{LR} $ rule,
we proceed by adding $  D$ to the appropriate weakening context of the
conclusion of that rule instance. To illustrate, if the last rule is $ \rightarrow_{R}$,
we transform:
 $$\frac{A\Rightarrow B}{\Gamma\Rightarrow A\rightarrow B}\rightarrow_{R}$$
into
 $$\frac{A\Rightarrow B}{\Gamma, D\Rightarrow A\rightarrow B}\rightarrow_{R}$$
and if the last rule is $ \rightarrow_{LR}$,
we transform:
$$\frac{A\Rightarrow B~~~B\Rightarrow A~~~C\Rightarrow E~~~E\Rightarrow C}{\Gamma , A\rightarrow C\Rightarrow B\rightarrow E}\rightarrow_{LR}$$
into
$$\frac{A\Rightarrow B~~~B\Rightarrow A~~~C\Rightarrow E~~~E\Rightarrow C}{\Gamma , D, A\rightarrow C\Rightarrow B\rightarrow E}\rightarrow_{LR} $$
\end{proof}

In the following, the notation $ \vdash_{n} \Gamma\Rightarrow C$ will stand for: the sequent $\Gamma\Rightarrow C$ in \gwf is derivable with a height of derivation
at most $ n $.

For proving the admissibility of contraction, we will need the following inversion lemma:

\begin{lem}\label{2}
In \gwf we have:
\begin{enumerate}
\item If $ \vdash_{n} A\wedge B, \Gamma \Rightarrow \Delta $, then $ \vdash_{n} A, B, \Gamma \Rightarrow \Delta $.
\item If $ \vdash_{n} A\vee B, \Gamma \Rightarrow \Delta $, then $ \vdash_{n} A, \Gamma \Rightarrow \Delta $ and $ \vdash_{n} B, \Gamma \Rightarrow \Delta $.
\item If $ \vdash_{n}  \Gamma \Rightarrow \Delta , A\wedge B $, then $ \vdash_{n}  \Gamma \Rightarrow \Delta,  A $ and $ \vdash_{n}  \Gamma \Rightarrow \Delta, B $.
\item If $ \vdash_{n}  \Gamma \Rightarrow \Delta , A\vee B $, then $ \vdash_{n}  \Gamma \Rightarrow \Delta,  A, B$.
\end{enumerate}
\end{lem}
\begin{proof}
The proof is by induction on $ n $ and similar to the intuitionistic case, (see, Theorem 2.3.5, \cite{Neg}).
\end{proof}
Next we prove the admissibility of the rule of contraction in {\sf GWF}:

\begin{thm}\label{1c}
\textbf{Height-preserving contraction.}
The left and right rules of contraction are height-preserving admissible in \gwf.
$$\frac{D, D, \Gamma\Rightarrow \Delta}{D, \Gamma\Rightarrow \Delta}L_{c}~~~~~~~~\frac{ \Gamma\Rightarrow \Delta, D, D}{\Gamma\Rightarrow \Delta, D}R_{c}$$
\end{thm}
\begin{proof}
The proof is by  induction on the height of  derivation of the premiss for left and right contraction. The base case is straightforward.
We have two cases according to whether the contraction formula is not principal or is principal in the last inference step.

If the contraction formula $ D $ is not principal in the last rule, then the proof goes through as for intuitionistic logic (see, Theorem 2.4.1, \cite{Neg}).

If the contraction formula $ D $ is principal in the last rule, we have three cases according to the form  of $ D $. If $ D=A\wedge B $ or $ D=A\vee B $, then the proof goes through as for intuitionistic logic (see, Theorem 2.4.1, \cite{Neg}).

Assume $ D=A\rightarrow B $. Then the last step is $\rightarrow_{LR}  $  or $ \rightarrow_{R} $. To illustrate, if the last rule is $\rightarrow_{LR}  $, then there exist formulas $E  $ and $ F $ such that $ \Delta=E\rightarrow F , \Delta^{'}$. In this case, we transform:
$$\frac{A\Rightarrow E~~~E\Rightarrow A~~~B\Rightarrow F~~~F\Rightarrow B}{\Gamma, A\rightarrow B , A\rightarrow B\Rightarrow E\rightarrow F , \Delta^{'}}\rightarrow_{LR}  $$
into
$$\frac{A\Rightarrow E~~~E\Rightarrow A~~~B\Rightarrow F~~~F\Rightarrow B}{\Gamma,  A\rightarrow B\Rightarrow E\rightarrow F , \Delta^{'}}\rightarrow_{LR} $$
A similar argument applies to all the other cases.
\end{proof}

\begin{dfn}
 The \textbf{cut-height}  of an instance of the rule of cut in a derivation is the sum of heights of derivation of the two premises of cut.
\end{dfn}

\begin{thm}\label{1d}
The rule of cut  is addmissible in {\sf GWF}.
$$ \frac{\Gamma\Rightarrow  D, \Delta~~~~~~D, \Gamma^{'}\Rightarrow \Delta^{'}}{\Gamma,\Gamma^{'} \Rightarrow \Delta, \Delta^{'}}Cut$$
\end{thm}

\begin{proof}
The proof is by induction on the weight of the cut formula $ D $ with a sub-induction on the cut-height.
The proof is structured as follows: We first address the case where at least one premise in a cut is an axiom or the conclusion of $ \bot_{L} $ and demonstrate the elimination of the cut. For the remaining cases, there are three cases: The cut formula is not principal in either premise of the cut, the cut formula is principal in only one premise of the cut, and the cut formula is principal in both premises of the cut.

In the derivations, it is assumed that the topsequents, from left to right, have derivation heights $n,m,k,\ldots$.

\textbf{Cut with an axiom or conclusion of $ \bot_{L} $ as premiss:} The proof is similar to the intuitionistic logic (see, Theorem 2.4.3, \cite{Neg}).

\textbf{Cut with neither premiss an axiom}: We have three cases:

1. Cut formula $ D $ is not principal in the left premiss, that is, not derived by an R-rule. We have the subcases  $ \wedge_{L} $, with $ \Gamma = A\wedge B, \Gamma^{''} $, and $ \vee_{L} $, with $ \Gamma = A\vee B, \Gamma^{''} $ according to the rule used to derive the left premiss. The proofs are similar to the intuitionistic logic (see, Theorem 2.4.3, \cite{Neg}).

2. Cut formula $ D $ is principal in the left premiss only, and the derivation is
transformed into one with a cut of lesser cut-height according to the derivation of
the right premiss. We have six subcases according to the rule used. The proofs of subcases $ \wedge_{L} $, with $ \Gamma^{'}= A\wedge B, \Gamma^{''} $,  $ \vee_{L}$, with $\Gamma^{'} = A\vee B, \Gamma^{''} $,  $ \wedge_{R} $, with $\Delta^{'}= A\wedge B, \Delta^{''}$, and  $ \vee_{R} $, with $\Delta^{'}= A\vee B, \Delta^{''}$ are similar to the intuitionistic logic.

2.5. $ \rightarrow_{R}$, with  $\Delta^{'}= A\rightarrow B, \Delta^{''}$ and the derivation with a cut of cut-height $ n+m+1 $
$$ \frac{\genfrac{}{}{0pt}{}{}{\Gamma\Rightarrow D, \Delta}~~~~\frac{A\Rightarrow B}{D, \Gamma^{'}\Rightarrow A\rightarrow B, \Delta^{''}}\rightarrow_{R}}{\Gamma, \Gamma^{'}\Rightarrow A\rightarrow B,\Delta, \Delta^{''}}Cut$$
is transformed into the derivation without cut
$$\frac{A\Rightarrow B}{\Gamma, \Gamma^{'}\Rightarrow A\rightarrow B , \Delta, \Delta^{''}}\rightarrow_{R} $$

2.6. $ \rightarrow_{LR}$, with  $\Gamma^{'}= A\rightarrow C, \Gamma^{''}$ and $\Delta^{'}= B\rightarrow D, \Delta^{''}$ and the derivation with a cut of cut-height $ n+max \lbrace m, k, l, s\rbrace +1 $
$$ \frac{\genfrac{}{}{0pt}{}{}{\Gamma\Rightarrow D, \Delta}~~~~\frac{A\Rightarrow B ~~B\Rightarrow A~~C\Rightarrow F~~F\Rightarrow C}{D, \Gamma^{''}, A\rightarrow C\Rightarrow B\rightarrow F, \Delta^{''}}\rightarrow_{LR}}{\Gamma, \Gamma^{''}, A\rightarrow C\Rightarrow B\rightarrow F, \Delta , \Delta^{''}}Cut$$
is transformed into the derivation without cut
$$ \frac{A\Rightarrow B ~~B\Rightarrow A~~C\Rightarrow F~~F\Rightarrow C}{\Gamma, \Gamma^{''}, A\rightarrow C\Rightarrow B\rightarrow F, \Delta , \Delta^{''}} \rightarrow_{LR}$$

3. Cut formula $ D $ is principal in both premises, and we have three subcases. The proofs for the cases where $ D=A\wedge B $ and $ D=A\vee B $ are  similar to the intuitionistic logic.

3.3. $ D=A\rightarrow B $ and the derivation is either
$$\frac{\frac{C\Rightarrow A~~~~A\Rightarrow C~~~~E\Rightarrow B~~~~B\Rightarrow E}{\Gamma, C\rightarrow E\Rightarrow A\rightarrow B}\rightarrow_{LR}~~~~\frac{A\Rightarrow F~~~~F\Rightarrow A~~~~B\Rightarrow G~~~~G\Rightarrow B}{A\rightarrow B\Rightarrow F\rightarrow G}\rightarrow_{LR}}{\Gamma, C\rightarrow E\Rightarrow F\rightarrow G}Cut$$
with cut-height $ max\lbrace n, m, k, l\rbrace +1+max\lbrace s, t, u, v\rbrace +1 $, or
  
  $$\frac{\frac{A\Rightarrow B}{ \Gamma \Rightarrow A\rightarrow B}\rightarrow_{R}~~~~~\frac{A\Rightarrow C~~~~C\Rightarrow A~~~~B\Rightarrow F~~~~F\Rightarrow B}{A\rightarrow B\Rightarrow C\rightarrow F}\rightarrow_{LR}}{\Gamma \Rightarrow C\rightarrow F}Cut$$
with  cut-height $n+1+ max\lbrace m, k, l, s\rbrace +1$, or

$$\frac{\frac{C\Rightarrow A~~~~A\Rightarrow C~~~~E\Rightarrow B~~~~B\Rightarrow E}{\Gamma, C\rightarrow E\Rightarrow A\rightarrow B}\rightarrow_{LR}~~~~\frac{F\Rightarrow G}{ A\rightarrow B\Rightarrow F\rightarrow G}\rightarrow_{R}}{\Gamma, C\rightarrow E\Rightarrow F\rightarrow G}Cut$$
with cut-height $ max\lbrace n, m, k, l\rbrace +1+s +1 $, or

  $$\frac{\frac{A\Rightarrow B}{ \Gamma \Rightarrow A\rightarrow B}\rightarrow_{R}~~~~~\frac{C\Rightarrow F}{A\rightarrow B\Rightarrow C\rightarrow F}\rightarrow_{R}}{\Gamma \Rightarrow C\rightarrow F}Cut$$
with cut-height $ n +1+m+1 $.

The first case transformed into the derivation with cuts of cut-heights $ n+s $, $ t+m $, $ k+u $ and $ v+l $   as follows:

$$ \frac{\frac{C\Rightarrow A~~~~~A\Rightarrow F}{C\Rightarrow F}Cut~~~~\frac{F\Rightarrow A~~~A\Rightarrow C}{F\Rightarrow C}Cut~~~~\frac{E\Rightarrow B~~~B\Rightarrow G}{E\Rightarrow G}Cut~~~~\frac{G\Rightarrow B~~~B\Rightarrow E}{G\Rightarrow E}Cut}{\Gamma, C\rightarrow E\Rightarrow F\rightarrow G}\rightarrow LR$$

and the second one transformed into the derivation with cuts of cut-heights $ k+n $  and $max\lbrace n, k \rbrace +1+l $   as follows:

$$\frac{\frac{C\Rightarrow A~~~~A\Rightarrow B}{C\Rightarrow B}Cut  ~~~~\genfrac{}{}{0pt}{}{}{B\Rightarrow F}}{\frac{C\Rightarrow F}{\Gamma \Rightarrow C\rightarrow F}\rightarrow_{R}}Cut $$
in the first cut, cut-height is reduced; in the first and second, weight of cut formula.

Transformation of the last two cases are also easy.
\end{proof}

\subsection{Equivalence with the axiomatic system and syntactic completeness}

In this section we establish the equivalence between the Hilbert-style system \wf and the corresponding sequent calculus, i.e. \gwf. As a consequence, we will show that \gwf is sound and
complete with respect to the appropriate class of neighbourhood frames.

We write $ {\sf  GWF }  \vdash \Gamma \Rightarrow \Delta $ if the sequent $ \Gamma \Rightarrow \Delta $ is derivable in \gwf, and
we say that $ A $ is derivable in \gwf whenever $ {\sf  GWF }  \vdash  \Rightarrow A$. 

For proving the equivalence between the Hilbert-style system \wf and the corresponding sequent calculus, we will need the following inversion lemma:

\begin{lem}\label{44}
If $ \gwf\vdash_{n}\Rightarrow A\rightarrow B$, then $ \gwf\vdash_{n} A\Rightarrow B$. 
\end{lem}
\begin{proof}
The proof is by induction on the height of the derivation of the premises.
There is only one possibility, i.e.; $ \Rightarrow A\rightarrow B $ is derived by the rule $ \rightarrow_{R} $. Then the premiss is what we wanted to derive.
\end{proof}

\begin{thm}\label{3}
Derivability in the sequent system \gwf and in the Hilbert-style  system \wf are
equivalent, i.e.
$${\sf  GWF }  \vdash \Gamma \Rightarrow \Delta ~~~~\mbox{iff}~~~~\vdash_{\wf} \bigwedge \Gamma \rightarrow \bigvee \Delta $$
\end{thm}

\begin{proof}
\textbf{ Left-to-right:} The proof proceeds by induction on the height of the derivation in {\sf  GWF}.
If the derivation has height $  0$, we have an axiom or an instance on $ \bot_{L} $. In both cases the claim holds. If the height is $n + 1$, we consider the last rule applied in the derivation.

We just prove for the rules $ \vee_{L} $ and $ \rightarrow_{LR} $ as follows, respectfully. The proof of other rules are similar and easy.

Assume the last derivation is by the rule $ \vee_{L}$. In this case we have ${\sf  GWF }  \vdash_{n+1} A\vee B, \Gamma \Rightarrow C$  (i.e, this sequent is derivable with a height of derivation $ n+1 $). We need to show that  $\vdash_{\wf}   \Gamma \bigwedge (A\vee B)  \rightarrow C $. By assumption we have $\vdash_{n} A , \Gamma \Rightarrow C  $ and  $\vdash_{n} B , \Gamma \Rightarrow C  $. By the induction hypothesis we conclude that  $\vdash_{\wf} A\bigwedge\Gamma \rightarrow C  $ and  $\vdash_{\wf}  B\\bigwedge \Gamma\rightarrow C  $. Using these results, we have the following proof in the Hilbert-style proof system for \wf:
\begin{enumerate}
\item $\vdash_{\wf} A\bigwedge\Gamma \rightarrow C $ ~~~~~~~~~~~by IH
\item $\vdash_{\wf}  B\bigwedge \Gamma\rightarrow C  $ ~~~~~~~~~~~by IH
\item $\vdash_{\wf}  (A\bigwedge\Gamma)\vee (B\bigwedge \Gamma)\rightarrow C  $~~~~~~~~~~~by 1,2 and rule 11
\item $\vdash_{\wf}   \Gamma \bigwedge (A\vee B)  \rightarrow (A\bigwedge\Gamma)\vee (B\bigwedge \Gamma)  $~~~~~~axiom 7
\item $\vdash_{\wf}   \Gamma \bigwedge (A\vee B)  \rightarrow C $~~~~~~~~~~~~~~~~~~by 3, 4 and rule 9
\end{enumerate}

Assume the last derivation is by the rule $ \rightarrow_{LR} $. In this case we have:
 $${\sf  GWF }  \vdash_{n+1}  \Gamma, A\rightarrow C\Rightarrow B\rightarrow D, \Delta.$$
We need to show that  $ \vdash_{\wf} \Gamma \bigwedge (A\rightarrow C) \rightarrow (B\rightarrow D) \bigvee \Delta$. By assumption we have:
 $$\vdash_{n} A   \Rightarrow B,  ~~\vdash_{n} B   \Rightarrow A,~~ \vdash_{n} C  \Rightarrow D,~~\vdash_{n} D\Rightarrow C .$$
Now, by the induction hypothesis we conclude that: 
 $$ \vdash_{\wf} A \leftrightarrow B~~~~~ and  ~~~~~\vdash_{H}C  \leftrightarrow D,$$
and so by the rule 13 of Figure \ref{fig:axioms:wf}, we have $ \vdash_{\wf} (A \rightarrow C)\leftrightarrow (B\rightarrow D) $. 

 Finally, since we have $ \vdash_{\wf} \Gamma \bigwedge (A\rightarrow C) \rightarrow    (A\rightarrow C)$ and $ \vdash_{\wf}  (B\rightarrow D) \rightarrow (B\rightarrow D) \bigvee \Delta$, we conclude that $ \vdash_{\wf}   \Gamma \bigwedge (A\rightarrow C) \rightarrow (B\rightarrow D) \bigvee \Delta$. 

\textbf{Right-to-left:} 
We argue by induction on the height of the axiomatic derivation in \wf. In order to prove the base case, we just need to show that all axioms  of the Hilbert-style system for \wf  (Figure \ref{fig:axioms:wf}) can be deduced
in \gwf. A straightforward application of the rules of the sequent calculus \gwf, possibly using Lemma \ref{1}. As an example we show the axiom 7 from Figure \ref{fig:axioms:wf} as follows  (the proof of other axioms are easy):

$$ \frac{\frac{\frac{\frac{\frac{\frac{\frac{•}{A, B\Rightarrow A}\ref{1}~~~\frac{•}{A, B\Rightarrow B}\ref{1}}{A, B\Rightarrow A\wedge B}\wedge_{R}}{A, B\Rightarrow A\wedge B, A\wedge C}R_{w}}{A, B\Rightarrow (A\wedge B)\vee (A\wedge C)}\vee_{R}~~~~~\frac{\frac{\frac{\frac{•}{A, C\Rightarrow C}\ref{1}~~~\frac{•}{A, C\Rightarrow A}\ref{1}}{A, C\Rightarrow A\wedge C}\wedge_{R}}{A, C\Rightarrow A\wedge B, A\wedge C}R_{w}}{A, C\Rightarrow (A\wedge B)\vee (A\wedge C)}\vee_{R}}{A, B\vee C\Rightarrow (A\wedge B)\vee (A\wedge C)}\vee_{L}}{(A\wedge (B\vee C))\Rightarrow (A\wedge B)\vee (A\wedge C)}\wedge_{L}}{\Rightarrow (A\wedge (B\vee C))\rightarrow (A\wedge B)\vee (A\wedge C)}\rightarrow_{R}$$

For the inductive steps, we need to show that the rules of the Hilbert-style system for \wf  (Figure \ref{fig:axioms:wf}), considering the Definition~\ref{hh}, can be deduced
in \gwf. As an example, if the last step is by the rule 13  from Figure \ref{fig:axioms:wf}, then $ \Gamma =\emptyset $ and $ \Delta $ is $ (A\rightarrow C)\leftrightarrow (B\rightarrow D) $. We know that we have derived $ (A\rightarrow C)\leftrightarrow (B\rightarrow D) $ from $ A\leftrightarrow B $ and $ C\leftrightarrow D $. Remember that $ A\leftrightarrow B $ and $ C\leftrightarrow D $ are defined as $ (A\rightarrow B)\wedge (B\rightarrow A) $ and $ (C\rightarrow D)\wedge (D\rightarrow C) $. Thus we assume, by induction hypothesis, that $ \gwf  \vdash \Rightarrow  (A\rightarrow B)\wedge (B\rightarrow A) $ and 
$ \gwf  \vdash \Rightarrow  (C\rightarrow D)\wedge (D\rightarrow C) $. By Lemma \ref{2}, we obtain that $  \gwf  \vdash \Rightarrow  A\rightarrow B$, $  \gwf  \vdash \Rightarrow  B\rightarrow A$, $  \gwf  \vdash \Rightarrow  C\rightarrow D$, $  \gwf  \vdash \Rightarrow  D\rightarrow C$. Again, by Lemma \ref{44}, we conclude $  \gwf  \vdash   A\Rightarrow B$, $  \gwf  \vdash   B\Rightarrow A$, $  \gwf  \vdash  C\Rightarrow  D$, $  \gwf  \vdash  D \Rightarrow C$. We can thus proceed as follows (for the sake of space, we divide the following proof into several smaller derivations):

$$ \frac{\frac{\frac{}{A\Rightarrow B}\!I\!H+\ref{2}~~\frac{•}{B\Rightarrow A}\!I\!H+\ref{2}~~\frac{•}{C\Rightarrow D}\!I\!H+\ref{2}~~\frac{•}{D\Rightarrow C}\!I\!H+\ref{2}}{(A\rightarrow C)\Rightarrow(B\rightarrow D)}\rightarrow_{LR}}{\Rightarrow (A\rightarrow C)\rightarrow (B\rightarrow D)}\rightarrow_{R}$$

$$\frac{\frac{\frac{}{C\Rightarrow D}\!I\!H+\ref{2}~~\frac{•}{D\Rightarrow C}\!I\!H+\ref{2}~~\frac{•}{B\Rightarrow A}\!I\!H+\ref{2}~~\frac{•}{A\Rightarrow B}\!I\!H+\ref{2}}{(B\rightarrow D)\Rightarrow(A\rightarrow C)}\rightarrow_{LR}}{\Rightarrow (B\rightarrow D)\rightarrow (A\rightarrow C)}\rightarrow_{R} $$

We now proceed to the next segment of the derivation,

$$\frac{\Rightarrow (A\rightarrow C)\rightarrow (B\rightarrow D)~~~~\Rightarrow (B\rightarrow D)\rightarrow (A\rightarrow C)}{\Rightarrow (A\rightarrow C)\leftrightarrow (B\rightarrow D)}\wedge_{R}$$
We can derive the rule 9 from Figure \ref{fig:axioms:wf} by the cut rule and considering that we have previously established the cut-admissibility theorem, therefore we are permitted to use cut rule in the proof of this theorem.
\end{proof}  

By combining Theorem \ref{4} and Theorem  \ref{3} we have the following result.

\begin{cor}
The calculus \gwf is sound and complete with respect to the class of all $ N\!B $-neighbourhood frames for \wf.
\end{cor}

\section{Sequent calculi for some subintuitionistic logics between \wf and \f}\label{some}

In this section, we will introduce sequent calculi for some of the logics that were presented in Section \ref{sec:modneigh}. 

We can extend \gwf by adding the following rules:

$$ \frac{A\Rightarrow C, B ~~~~A, D\Rightarrow B~~~~C\Rightarrow A, D~~~~C, B\Rightarrow D}{\Gamma, C\rightarrow D\Rightarrow A\rightarrow B, \Delta}\rightarrow_{LR_{N}} $$

$$\frac{A\Rightarrow C, B~~~~A, D\Rightarrow B}{\Gamma, C\rightarrow D\Rightarrow A\rightarrow B, \Delta}\rightarrow_{LR_{N_{2}}}$$

$$  \frac{A\Rightarrow B }{\Gamma,  C\rightarrow A \Rightarrow  C \rightarrow B, \Delta}\rightarrow_{\widehat{\sf C}} $$

$$  \frac{A\Rightarrow B }{\Gamma, B\rightarrow C\Rightarrow A \rightarrow C, \Delta}\rightarrow_{ \widehat{\sf D}}$$

If $ \sf {X} $ is one of the above rules, then $ \sf {GWFX }$ is obtained by adding rule $\sf{ X} $ to the sequent calculus \gwf.

First we prove the admissibility of weakening and contraction for ${\sf  GWFX}$.

\begin{lem}\label{1N}
The sequent $ C, \Gamma\Rightarrow \Delta, C $ is derivable in ${\sf  GWFX }$ for an arbitrary formula $ C $ and arbitrary contexts $ \Gamma $ and $\Delta$.
\end{lem}
\begin{proof}
The proof is similar to the proof of lemma \ref{1}.
\end{proof}

\begin{thm}
\textbf{Height-preserving weakening.} 
The left and right rules of weakening are height-preserving admissible in ${\sf  GWFX }$.
$$ \frac{\Gamma\Rightarrow \Delta}{A, \Gamma\Rightarrow \Delta}L_{w}~~~~~~~~~~~~~\frac{\Gamma\Rightarrow \Delta}{\Gamma\Rightarrow \Delta, A}R_{w}$$
\end{thm}
\begin{proof}
The proof is similar to the proof of Theorem \ref{1b}.
\end{proof}

\begin{lem}\label{23a}
In ${\sf  GWFX }$ we have:
\begin{enumerate}
\item If $ \vdash_{n} A\wedge B, \Gamma \Rightarrow \Delta $, then $ \vdash_{n} A, B, \Gamma \Rightarrow \Delta $.
\item If $ \vdash_{n} A\vee B, \Gamma \Rightarrow \Delta $, then $ \vdash_{n} A, \Gamma \Rightarrow \Delta $ and $ \vdash_{n} B, \Gamma \Rightarrow \Delta $.
\item If $ \vdash_{n}  \Gamma \Rightarrow \Delta , A\wedge B $, then $ \vdash_{n}  \Gamma \Rightarrow \Delta,  A $ and $ \vdash_{n}  \Gamma \Rightarrow \Delta, B $.
\item If $ \vdash_{n}  \Gamma \Rightarrow \Delta , A\vee B $, then $ \vdash_{n}  \Gamma \Rightarrow \Delta,  A, B$.
\end{enumerate}
\end{lem}
\begin{proof}
Similar to the proof of Lemma \ref{2}. 
\end{proof}

\begin{thm}
\textbf{Height-preserving contraction.}
The left and right rules of contraction are height-preserving admissible in ${\sf  GWFX }$.
$$\frac{D, D, \Gamma\Rightarrow \Delta}{D, \Gamma\Rightarrow \Delta}L_{c}~~~~~~~~\frac{ \Gamma\Rightarrow \Delta, D, D}{\Gamma\Rightarrow \Delta, D}R_{c}$$
\end{thm}
\begin{proof}
The proof is similar to the proof of Theorem \ref{1c}.
\end{proof}

\subsection{The system ${\sf  GWF_{N } }$}\label{modneigh2}

If we consider $\rightarrow_{LR_{N}} $ instead of $\rightarrow_{ LR}$ in the system ${\sf  GWF }$, then the  system ${\sf  GWF_{N } }$ is defined.
We prove the admissibility of cut for ${\sf  GWF_{N }}$.

\begin{thm}
The rule of cut  is addmissible in ${\sf GWF_{N}}$.
$$ \frac{\Gamma\Rightarrow  D, \Delta~~~~~~D, \Gamma^{'}\Rightarrow \Delta^{'}}{\Gamma,\Gamma^{'} \Rightarrow \Delta, \Delta^{'}}Cut$$
\end{thm}
\begin{proof}
By theorem \ref{1d}, we just need to consider the following cases.

Both premises are by rule $ \rightarrow_{LR_{N}} $. We transform

$$\frac{\frac{A\Rightarrow D, B~~~~A, E\Rightarrow B~~~~D\Rightarrow A, E~~~~D, B\Rightarrow E}{\Gamma, D\rightarrow E\Rightarrow A\rightarrow B}\rightarrow_{LR_{N}}~~~~\frac{F\Rightarrow A, G~~~~F, B\Rightarrow G~~~~A\Rightarrow F, B~~~~A, G\Rightarrow B}{A\rightarrow B\Rightarrow F\rightarrow G}\rightarrow_{LR_{N}}}{\Gamma, D\rightarrow E\Rightarrow F\rightarrow G}Cut$$

 into the following derivation (for the sake of space, we divide the following proof into several smaller derivations)

$$ \frac{\frac{\frac{\frac{F\Rightarrow A, G~A\Rightarrow D, B}{F\Rightarrow G, D, B}Cut~ \genfrac{}{}{0pt}{}{}{ F, B\Rightarrow   G}}{F, F\Rightarrow D, G, G}Cut}{F, F\Rightarrow D, G}R_{c}}{F\Rightarrow D, G}L_{c} $$

$$ \frac{\frac{\frac{\frac{F\Rightarrow A, G~A, E\Rightarrow  B}{F, E\Rightarrow G, B}Cut~ \genfrac{}{}{0pt}{}{}{ F, B\Rightarrow   G}}{F, F, E\Rightarrow G, G}Cut}{F, F, E\Rightarrow G}R_{c}}{F, E\Rightarrow  G}L_{c} $$

$$ \frac{\frac{\frac{\frac{D\Rightarrow A, E~A \Rightarrow  F, B}{D\Rightarrow F, B, E}Cut~ \genfrac{}{}{0pt}{}{}{ D, B\Rightarrow   E}}{D,D\Rightarrow F, E, E}Cut}{D, D\Rightarrow F, E}R_{c}}{D\Rightarrow  F, E}L_{c} $$

$$ \frac{\frac{\frac{\frac{D\Rightarrow A, E~A, G\Rightarrow  B}{D, G\Rightarrow E, B}Cut~ \genfrac{}{}{0pt}{}{}{ D, B\Rightarrow   E}}{D, D, G\Rightarrow E, E}Cut}{D, D, G\Rightarrow E}R_{c}}{D, G\Rightarrow  E}L_{c} $$

We now proceed to the next segment of the derivation,

$$\frac{F\Rightarrow D, G~~~~F, E\Rightarrow  G~~~~D\Rightarrow  F, E~~~~D, G\Rightarrow  E}{\Gamma, D\rightarrow E\Rightarrow F\rightarrow G}\rightarrow_{LR_{N}}  $$

Left premiss is by $ \rightarrow_{R} $, and right one by $ \rightarrow_{LR_{N}} $. We transform
 $$\frac{\frac{A\Rightarrow B}{ \Gamma \Rightarrow A\rightarrow B}\rightarrow_{R}~~~~~\frac{F\Rightarrow A, G~~~~F, B\Rightarrow G~~~~A\Rightarrow F, B~~~~A, G\Rightarrow B}{A\rightarrow B\Rightarrow F\rightarrow G}\rightarrow_{LR_{N}}}{\Gamma \Rightarrow F\rightarrow G}Cut$$

\begin{center}
into 
\end{center}

$$\frac{\frac{ \genfrac{}{}{0pt}{}{}{ F\Rightarrow A, G}~~~\frac{A\Rightarrow B~~F, B\Rightarrow G}{A, F\Rightarrow G}Cut}{\frac{F, F\Rightarrow G, G}{\frac{F, F\Rightarrow G}{F\Rightarrow G}L_{c}}R_{c}}Cut}{\Gamma\Rightarrow F\rightarrow G}\rightarrow_{R} $$

Left premiss is by $ \rightarrow_{LR_{N}} $, and right one by $ \rightarrow_{R} $. We transform
$$\frac{\frac{F\Rightarrow A, G~~~~F, B\Rightarrow G~~~~A\Rightarrow F, B~~~~A, G\Rightarrow B}{\Gamma, A\rightarrow B\Rightarrow F\rightarrow G}\rightarrow_{LR_{N}}~~~~\frac{C\Rightarrow D}{ F\rightarrow G\Rightarrow C\rightarrow D}\rightarrow_{R}}{\Gamma, A\rightarrow B\Rightarrow C\rightarrow D}Cut$$

\begin{center}
into 
\end{center}
$$ \frac{C\Rightarrow D}{\Gamma, A\rightarrow B\Rightarrow C\rightarrow D}\rightarrow_{R}$$
\end{proof}

\begin{lem}\label{45}
If $ {\sf  GWF_{N }  }\vdash_{n}\Rightarrow A\rightarrow B$, then ${\sf  GWF_{N }  } \vdash_{n} A\Rightarrow B$. 
\end{lem}
\begin{proof}
Similar to the proof of Lemma \ref{44}.
\end{proof}

Finally, we establish the equivalence between the Hilbert-style system $ {\sf  WF_{N }  } $ and the corresponding sequent calculus, i.e. ${\sf  GWF_{N }  }$.

\begin{thm} \label{3b}
Derivability in the sequent system ${\sf  GWF_{N }  }$ and in the Hilbert-style  system ${\sf  WF_{N }  }$ are
equivalent, i.e.
$${\sf  GWF_{N} }  \vdash \Gamma \Rightarrow \Delta ~~~~\mbox{iff}~~~~\vdash_{\sf  WF_{N }  } \bigwedge \Gamma \rightarrow \bigvee \Delta $$
\end{thm}

\begin{proof}
\textbf{Left-to-right:} By Theorem \ref{3}, we just need to assume the last derivation is by the rule $ \rightarrow_{LR_{N}} $. 
 In this case we have ${\sf  GWF_{N} }  \vdash_{n} \Gamma, C\rightarrow D \Rightarrow A\rightarrow B, \Delta$. We need to show that  $\vdash_{\sf WF_N}   \Gamma \bigwedge (C\rightarrow D)  \rightarrow (A\rightarrow B)\bigvee \Delta $. By assumption we have $\vdash_{n-1} A \Rightarrow C, B  $, $\vdash_{n-1} A , D\Rightarrow B  $, $\vdash_{n-1} C \Rightarrow A, D $ and  $\vdash_{n-1} C, B \Rightarrow D  $. By the induction hypothesis we conclude that  $\vdash_{\sf WF_N} A\rightarrow C\vee B  $, $\vdash_{\sf WF_N} A\wedge D\rightarrow  B  $, $\vdash_{\sf WF_N} C\rightarrow A\vee D$ and  $\vdash_{\sf WF_N} C\wedge B\rightarrow D$. Using these results, we have the following proof in the Hilbert-style proof system for $  {\sf WF_N}$:
\begin{enumerate}
\item  $\vdash_{\sf WF_N} A\rightarrow C\vee B  $~~~~~~~~~~~~~~~~by IH
\item $\vdash_{\sf WF_N} A\wedge D\rightarrow  B  $~~~~~~~~~~~~~~~~by IH
\item $\vdash_{\sf WF_N} C\rightarrow A\vee D$~~~~~~~~~~~~~~~~by IH
\item  $\vdash_{\sf WF_N} C\wedge B\rightarrow D$~~~~~~~~~~~~~~~~by IH
\item $ \vdash_{\sf WF_N} (C\rightarrow D)\rightarrow (A\rightarrow B) $~~~~~~~~~~by 1, 2, 3, 4 and rule N
\item $\vdash_{\sf WF_N} \Gamma\bigwedge  (C\rightarrow D) \rightarrow (C\rightarrow D)  $ ~~~~by axiom 4
\item $\vdash_{\sf WF_N} \Gamma\bigwedge (C\rightarrow D) \rightarrow (A\rightarrow B)  $~~~~ by 5, 6 and rule 9
\item $\vdash_{\sf WF_N} (A\rightarrow B)\rightarrow (A\rightarrow B)\bigvee\Delta  $~~~~~by axiom 1
\item $\vdash_{\sf WF_N} \Gamma\bigwedge(C\rightarrow D) \rightarrow (A\rightarrow B)\bigvee\Delta  $ ~~~~by 7, 8 and rule 9
\end{enumerate}

\textbf{Right-to-left:} By Theorem \ref{3}, we just need to assume the last step is by the rule $ N $, then $ \Gamma =\emptyset $ and $ \Delta $ is $ (A\rightarrow B)\leftrightarrow (C\rightarrow D) $. We know that we have derived $ (A\rightarrow B)\leftrightarrow (C\rightarrow D) $ from $ A\rightarrow B\vee C$, $ C\rightarrow A\vee D $, $ A\wedge D\rightarrow B $ and $ C\wedge B\rightarrow D $. Thus we assume, by induction hypothesis, that $ {\sf  GWF_{N }  }  \vdash \Rightarrow  A\rightarrow B\vee C $,
$ {\sf  GWF_{N }  } \vdash \Rightarrow  C\rightarrow A\vee D $, ${\sf  GWF_{N }  }  \vdash \Rightarrow    A\wedge D\rightarrow B $ and $ {\sf  GWF_{N }  }  \vdash \Rightarrow  C\wedge B\rightarrow D $. By Lemmas \ref{23a} and \ref{45}, we obtain that $ {\sf  GWF_{N }  }  \vdash  A\Rightarrow  B, C $,
$ {\sf  GWF_{N }  } \vdash  C\Rightarrow  A, D $, ${\sf  GWF_{N }  }  \vdash A, D\Rightarrow B $ and $ {\sf  GWF_{N }  }  \vdash   C, B\Rightarrow D $. We can thus proceed as follows (for the sake of space, we divide the following proof into several smaller derivations):

$$ \frac{\frac{\frac{•}{C\Rightarrow A, D}\!I\!H+\ref{23a}+\ref{45}~~\frac{•}{C, B\Rightarrow D}\!I\!H+\ref{23a}+\ref{45}~~\frac{•}{A\Rightarrow C, B}\!I\!H+\ref{23a}+\ref{45}~~\frac{•}{A, D\Rightarrow B}\!I\!H+\ref{23a}+\ref{45}}{ A\rightarrow B\Rightarrow C\rightarrow D}\!\rightarrow_{LR_{N}}}{\Rightarrow(A\rightarrow B)\rightarrow (C\rightarrow D)}\!\rightarrow_{R} $$

$$ \frac{\frac{\frac{•}{A\Rightarrow C, B}\!I\!H+\ref{23a}+\ref{45}~~\frac{•}{A, D\Rightarrow B}\!I\!H+\ref{23a}+\ref{45}~~\frac{•}{C\Rightarrow A, D}\!I\!H+\ref{23a}+\ref{45}~~\frac{•}{C, B\Rightarrow D}\!I\!H+\ref{23a}+\ref{45}}{ C\rightarrow D\Rightarrow A\rightarrow B}\!\rightarrow_{LR_{N}} }{\Rightarrow(C\rightarrow D)\rightarrow (A\rightarrow B)}\!\rightarrow_{R} $$

We now proceed to the next segment of the derivation,

$$ \frac{\Rightarrow(A\rightarrow B)\rightarrow (C\rightarrow D)~~~~~\Rightarrow(C\rightarrow D)\rightarrow (A\rightarrow B)}{\Rightarrow (A\rightarrow B)\leftrightarrow (C\rightarrow D)}\!\wedge_{R}$$

\end{proof}
By combining Theorem \ref{4b} and Theorem  \ref{3b} we have the following result.
\begin{cor}
The calculus ${\sf  GWF_{N }  }$ is sound and complete with respect to the class of $ N\!B $-neighbourhood frames that are closed under equivalence.
\end{cor}

\subsection{ The system ${\sf  GWF_{N_{2} } }$}\label{modneigh3}

If we consider $\rightarrow_{LR_{N_{2}}}$ instead of $\rightarrow_{ LR}$ in the system ${\sf  GWF_{N} }$, then the  system ${\sf  GWF_{N_{2} } }$ is defined.

First we prove the admissibility of cut for ${\sf  GWF_{N_{2} }}$.
\begin{thm}
The rule of cut  is addmissible in ${\sf GWF_{N_{2}}}$.
$$ \frac{\Gamma\Rightarrow  D, \Delta~~~~~~D, \Gamma^{'}\Rightarrow \Delta^{'}}{\Gamma,\Gamma^{'} \Rightarrow \Delta, \Delta^{'}}Cut$$
\end{thm}
\begin{proof}
By theorem \ref{1d}, we just need to consider the following cases.
Both premises are by rule $ \rightarrow_{LR_{N_{2}}} $. We transform
$$\frac{\frac{A\Rightarrow D, B~~~~A, E\Rightarrow B}{\Gamma, D\rightarrow E\Rightarrow A\rightarrow B}\rightarrow_{LR_{N_{2}}}~~~~\frac{F\Rightarrow A, G~~~~F, B\Rightarrow G}{A\rightarrow B\Rightarrow F\rightarrow G}\rightarrow_{LR_{N_{2}}}}{\Gamma, D\rightarrow E\Rightarrow F\rightarrow G}Cut$$
\begin{center}
into 
\end{center}
$$ \frac{\frac{\frac{\frac{\frac{F\Rightarrow A, G~A\Rightarrow D, B}{F\Rightarrow G, D, B}Cut~ \genfrac{}{}{0pt}{}{}{ F, B\Rightarrow   G}}{F, F\Rightarrow D, G, G}Cut}{F, F\Rightarrow D, G}R_{c}}{F\Rightarrow D, G}L_{c}~~~~\frac{\frac{\frac{\frac{F\Rightarrow A, G~A, E\Rightarrow  B}{F, E\Rightarrow G, B}Cut~ \genfrac{}{}{0pt}{}{}{ F, B\Rightarrow   G}}{F, F, E\Rightarrow G, G}Cut}{F, F, E\Rightarrow G}R_{c}}{F, E\Rightarrow  G}L_{c}}{\Gamma, D\rightarrow E\Rightarrow F\rightarrow G}\rightarrow_{LR_{N}}$$
Left premiss is by $ \rightarrow_{R} $, and right one by $ \rightarrow_{LR_{N_{2}}} $. We transform
 $$\frac{\frac{A\Rightarrow B}{ \Gamma \Rightarrow A\rightarrow B}\rightarrow_{R}~~~~~\frac{D\Rightarrow F, A~~~~ B, D\Rightarrow F}{A\rightarrow B\Rightarrow D\rightarrow F}\rightarrow_{LR_{N_{2}}}}{\Gamma \Rightarrow D\rightarrow F}Cut$$
\begin{center}
into 
\end{center}
$$\frac{\frac{ \genfrac{}{}{0pt}{}{}{D\Rightarrow F, A}~~~\frac{A\Rightarrow B~~B, D\Rightarrow F}{A, D\Rightarrow F}Cut}{\frac{D, D\Rightarrow F, F}{\frac{D, D\Rightarrow D}{D\Rightarrow D}L_{c}}R_{c}}Cut}{\Gamma\Rightarrow D\rightarrow F}\rightarrow_{R} $$
Left premiss is by $ \rightarrow_{LR_{N_{2}}} $, and right one by $ \rightarrow_{R} $. We transform
$$\frac{\frac{F\Rightarrow A, G~~~~F, B\Rightarrow G}{\Gamma, A\rightarrow B\Rightarrow F\rightarrow G}\rightarrow_{LR_{N_{2}}}~~~~\frac{C\Rightarrow D}{ F\rightarrow G\Rightarrow C\rightarrow D}\rightarrow_{R}}{\Gamma, A\rightarrow B\Rightarrow C\rightarrow D}Cut$$
\begin{center}
into 
\end{center}
$$ \frac{C\Rightarrow D}{\Gamma, A\rightarrow B\Rightarrow C\rightarrow D}\rightarrow_{R}$$
\end{proof}

\begin{lem}
If $ {\sf  GWF_{N_{2} }  }\vdash_{n}\Rightarrow A\rightarrow B$, then ${\sf  GWF_{N_{2} }  } \vdash_{n} A\Rightarrow B$. 
\end{lem}
\begin{proof}
Similar to the proof of Lemma \ref{44}.
\end{proof}

Now we establish the equivalence between the Hilbert-style system $ {\sf  WF_{N_{2} }  } $ and the corresponding sequent calculus, i.e. ${\sf  GWF_{N_{2} }  }$.

\begin{thm} \label{3c2}
Derivability in the sequent system ${\sf  GWF_{N_{2} }  }$ and in the Hilbert-style  system ${\sf  WF_{N_{2} }  }$ are
equivalent, i.e.
$${\sf  GWF_{N_{2} }  }  \vdash \Gamma \Rightarrow \Delta ~~~~\mbox{iff}~~~~\vdash_{\sf  WF_{N_{2} }  } \bigwedge \Gamma \rightarrow \bigvee \Delta $$
\end{thm}
\begin{proof}
The proof is similar to Theorem \ref{3}.
\end{proof}

By combining Theorem \ref{4b} and Theorem  \ref{3c2} we have the following result.
\begin{cor}
The calculus ${\sf  GWF_{N_{2} }  }$ is sound and complete with respect to the class of $ N\!B $-neighbourhood frames that are closed under superset equivalence.
\end{cor}

\subsection{The system ${\sf  GWF  \widehat{\sf D}}$}\label{modneigh6}

The system ${\sf  GWF  \widehat{\sf D}}$ is defined as $\gwf\oplus Cut \oplus \rightarrow_{\widehat{\sf D}}$.

The proof of the admissibility of the cut rule for the system ${\sf  GWF\widehat{\sf D}  }$ remains an open problem.

\begin{lem}\label{49}

If ${\sf  GWF\widehat{\sf D} } \vdash_{n}\Rightarrow A\rightarrow B$, then ${\sf  GWF\widehat{\sf D} } \vdash_{n} A\Rightarrow B$. 
\end{lem}
\begin{proof}
Similar to the proof of Lemma \ref{44}.
\end{proof}

Now we establish the equivalence between the Hilbert-style system ${\sf WF\widehat{\sf D}}$ and the corresponding sequent calculus, i.e. ${\sf  GWF\widehat{\sf D} } $.

\begin{thm}\label{3cdd}
Derivability in the sequent system ${\sf  GWF\widehat{\sf D} } $ and in the Hilbert-style  system $ {\sf WF\widehat{\sf D}} $ are
equivalent, i.e.
$${\sf  GWF\widehat{\sf D} }  \vdash \Gamma \Rightarrow \Delta ~~~~\mbox{iff}~~~~\vdash_{\sf WF\widehat{\sf D}} \bigwedge \Gamma \rightarrow \bigvee \Delta $$
\end{thm}

\begin{proof}
\textbf{Left-to-right:} By Theorem \ref{3}, we just need to assume the last derivation is by the rule $ \rightarrow_{\widehat{\sf D}} $. 
 In this case we have ${\sf  GWF\widehat{\sf D} }  \vdash_{n} \Gamma, B\rightarrow C\Rightarrow A\rightarrow C, \Delta$. We need to show that  $\vdash_{\sf WF\widehat{\sf D}}   \Gamma \bigwedge (B \rightarrow C)\rightarrow (A\rightarrow C)\bigvee \Delta $. By assumption we have $\vdash_{n-1}A \Rightarrow  B $. By the induction hypothesis we conclude that  $\vdash_{\sf WF\widehat{\sf D}} A \rightarrow B$. Using these results, we have the following proof in the Hilbert-style proof system for $  {\sf WF\widehat{\sf D}}$:
 \begin{enumerate}
\item  $\vdash_{\sf WF\widehat{\sf D}} A \rightarrow  B   $~~~~~~~~~~~~~~~~by IH
\item $\vdash_{\sf WF\widehat{\sf D}} A\vee B\leftrightarrow B $
\item $\vdash_{\sf WF\widehat{\sf D}} C\leftrightarrow C $
\item $\vdash_{\sf WF\widehat{\sf D}} (A\vee B\rightarrow C)\leftrightarrow(B\rightarrow C)$
\item $\vdash_{\sf WF\widehat{\sf D}} (A\vee B\rightarrow C)\rightarrow (A\rightarrow C)\wedge (B\rightarrow C)$~~~~~axiom $\widehat{\sf D}$
\item $\vdash_{\sf WF\widehat{\sf D}}  (B\rightarrow C)\rightarrow (A\rightarrow C)\wedge (B\rightarrow C)$~~~~~~~~~~by 4 and 5
\item $\vdash_{\sf WF\widehat{\sf D}}  (A\rightarrow C)\wedge (B\rightarrow C)\rightarrow (A\rightarrow C)$
\item $\vdash_{\sf WF\widehat{\sf D}}  (B\rightarrow C)\rightarrow (A\rightarrow C)$~~~~~~~~~by 6 and 7
\item $\vdash_{\sf WF\widehat{\sf D}}\Gamma\bigwedge (B\rightarrow C) \rightarrow (A\rightarrow C)\bigvee \Delta  $ 
\end{enumerate}

 \textbf{Right-to-left:} By Theorem \ref{3}, we just need to show that the axiom $ \widehat{\sf D} $ can be deduced in ${\sf GWF\widehat{\sf D}}$ as follows:
 
 $$\frac{\frac{\frac{\frac{\frac{•}{A\Rightarrow A, B}\ref{1N}}{A\Rightarrow A\vee B}\vee_{R}}{A\vee B\rightarrow C\Rightarrow A\rightarrow C}\rightarrow_{\widehat{\sf D}}~~~\frac{\frac{\frac{•}{B\Rightarrow A, B}\ref{1N}}{B\Rightarrow A\vee B}\vee_{R}}{A\vee B\rightarrow C\Rightarrow B\rightarrow C}\rightarrow_{\widehat{\sf D}}}{(A\vee B\rightarrow C)\Rightarrow (A\rightarrow C)\wedge (B\rightarrow C)}\wedge_{R}}{\Rightarrow (A\vee B\rightarrow C)\rightarrow (A\rightarrow C)\wedge (B\rightarrow C)}\rightarrow_{R}$$
\end{proof}

By combining Theorem \ref{4b} and Theorem  \ref{3cdd} we have the following result.
\begin{cor}
The calculus ${\sf  GWF\widehat{\sf D}  }$ is sound and complete with respect to the class of  $ N\!B $-neighbourhood frames that are closed under downset.
\end{cor}

\subsection{The system ${\sf  GWF  \widehat{\sf C}}$}\label{modneigh7}

The system ${\sf  GWF  \widehat{\sf C}}$ is defined as $\gwf\oplus Cut  \oplus\rightarrow_{\widehat{\sf C}}$.

The proof of the admissibility of the cut rule for the system ${\sf  GWF\widehat{\sf C}  }$ remains an open problem.

\begin{lem}\label{50}

If ${\sf  GWF\widehat{\sf C} } \vdash_{n}\Rightarrow A\rightarrow B$, then ${\sf  GWF\widehat{\sf C} } \vdash_{n} A\Rightarrow B$. 
\end{lem}
\begin{proof}
Similar to the proof of Lemma \ref{44}.
\end{proof}

Now we establish the equivalence between the Hilbert-style system ${\sf WF\widehat{\sf C}}$ and the corresponding sequent calculus, i.e. ${\sf  GWF\widehat{\sf C} } $.

\begin{thm}\label{3cdc}
Derivability in the sequent system ${\sf  GWF\widehat{\sf C} } $ and in the Hilbert-style  system $ {\sf WF\widehat{\sf C}} $ are
equivalent, i.e.
$${\sf  GWF\widehat{\sf C} }  \vdash \Gamma \Rightarrow \Delta ~~~~\mbox{iff}~~~~\vdash_{\sf WF\widehat{\sf C}} \bigwedge \Gamma \rightarrow \bigvee \Delta $$
\end{thm}
\begin{proof}
\textbf{Left-to-right:} By Theorem \ref{3}, we just need to assume the last derivation is by the rule $ \rightarrow_{\widehat{\sf C}} $. 
 In this case we have ${\sf  GWF\widehat{\sf C} }  \vdash_{n} \Gamma, C\rightarrow A\Rightarrow C\rightarrow B, \Delta$. We need to show that  $\vdash_{\sf WF\widehat{\sf C}}   \Gamma \bigwedge (C \rightarrow A)\rightarrow (C\rightarrow B)\bigvee \Delta $. By assumption we have $\vdash_{n-1}A \Rightarrow  B $. By the induction hypothesis we conclude that  $\vdash_{\sf WF\widehat{\sf C}} A \rightarrow B$. Using these results, we have the following proof in the Hilbert-style proof system for $  {\sf WF\widehat{\sf C}}$:
 \begin{enumerate}
\item  $\vdash_{\sf WF\widehat{\sf C}} A \rightarrow  B   $~~~~~~~~~~~~~~~~by IH
\item $\vdash_{\sf WF\widehat{\sf C}} A \leftrightarrow A\wedge B $
\item $\vdash_{\sf WF\widehat{\sf C}} C\leftrightarrow C $
\item $\vdash_{\sf WF\widehat{\sf C}} (C\rightarrow A)\leftrightarrow(C\rightarrow A\wedge B)$
\item $\vdash_{\sf WF\widehat{\sf C}} (C\rightarrow  A\wedge B)\rightarrow (C\rightarrow A)\wedge (C\rightarrow B)$~~~~~axiom $\widehat{\sf C}$
\item $\vdash_{\sf WF\widehat{\sf C}}  (C\rightarrow A)\rightarrow (C\rightarrow A)\wedge (C\rightarrow B)$~~~~~~~~~~by 4 and 5
\item $\vdash_{\sf WF\widehat{\sf C}}  (C\rightarrow A)\wedge (C\rightarrow B)\rightarrow (C\rightarrow B)$
\item $\vdash_{\sf WF\widehat{\sf C}}  (C\rightarrow A)\rightarrow (C\rightarrow B)$~~~~~~~~~~~~by 6 and 7
\item $\vdash_{\sf WF\widehat{\sf C}}\Gamma\bigwedge (C\rightarrow A) \rightarrow (C\rightarrow B)\bigvee \Delta  $ 
\end{enumerate}

 \textbf{Right-to-left:} By Theorem \ref{3}, we just need to show that the axiom $ \widehat{\sf C} $ can be deduced in ${\sf GWF\widehat{\sf C}}$ as follows:
 
 $$\frac{\frac{\frac{\frac{\frac{•}{B, C\Rightarrow  B}\ref{1N}}{B\wedge C\Rightarrow  B}\wedge_{L}}{A\rightarrow B\wedge C\Rightarrow A\rightarrow B}\rightarrow_{\widehat{\sf C}}~~~\frac{\frac{\frac{•}{B, C\Rightarrow C}\ref{1N}}{B\wedge C\Rightarrow C}\wedge_{L}}{A \rightarrow B\wedge C\Rightarrow A\rightarrow C}\rightarrow_{\widehat{\sf C}}}{(A \rightarrow B\wedge C)\Rightarrow (A\rightarrow B)\wedge (A\rightarrow C)}\wedge_{R}}{\Rightarrow (A\rightarrow B\wedge C)\rightarrow (A\rightarrow B)\wedge (A\rightarrow C)}\rightarrow_{R}$$
\end{proof}

By combining Theorem \ref{4b} and Theorem  \ref{3cdc} we have the following result.
\begin{cor}
The calculus ${\sf  GWF\widehat{\sf C}  }$ is sound and complete with respect to the class of  $ N\!B $-neighbourhood frames that are closed under upset.
\end{cor}

\section{Sequent calculus for some other subintuitionistic logics}\label{fcdi}

As previously stated in section \ref{sec:modneigh}, the logic \f is the smallest set of formulas closed under instances of \wf, {\sf C}, {\sf D} and {\sf I} \cite{FD}. In this section, in order to introduce sequent calculus for logic \f, we will introduce sequent calculus for ${\sf  WFI  }$, ${\sf  WFC }$, ${\sf  WFD}$, ${\sf  WFCI }$ and ${\sf  WFDI}$.
\subsection{The system ${\sf  GWFI } $}\label{modneigh6}

\begin{dfn}
The rules of the calculus \gwfi for subintuitionistic logic {\sf WFI} are the following ($ p $ atomic):

$$ \dfrac{•}{p, \Gamma\Rightarrow p}Ax~~~~~~~~~~~~~~~~~~~~~~ \dfrac{•}{\bot, \Gamma\Rightarrow C}\bot_{L}$$

\vspace{0.3cm}
$$\dfrac{A, B, \Gamma\Rightarrow C}{A\wedge B, \Gamma\Rightarrow C}\wedge_{L}~~~~~~~~~~~~~\dfrac{\Gamma\Rightarrow  A~~~\Gamma\Rightarrow   B}{ \Gamma\Rightarrow  A\wedge B}\wedge_{R}$$

\vspace{0.3cm}
$$\dfrac{A, \Gamma\Rightarrow C~~~B, \Gamma\Rightarrow C}{A\vee B, \Gamma\Rightarrow C}\vee_{L}~~~~~~~~~~~~~\dfrac{ \Gamma\Rightarrow  A}{ \Gamma\Rightarrow  A\vee B}\vee^{1}_{R}~~~\dfrac{ \Gamma\Rightarrow  B}{ \Gamma\Rightarrow  A\vee B}\vee^{2}_{R}$$

\vspace{0.3cm}
$$ \dfrac{A\Rightarrow B}{\Gamma\Rightarrow A\rightarrow B}\rightarrow_{R}~~~~~~~~~~\frac{A\Rightarrow B~~~B\Rightarrow A~~~C\Rightarrow D~~~D\Rightarrow C}{\Gamma, A\rightarrow C\Rightarrow B\rightarrow D}\rightarrow_{LR}$$

\vspace{0.3cm}
$$  \frac{\Gamma\Rightarrow B\rightarrow C~~~~~\Gamma\Rightarrow C\rightarrow D}{\Gamma\Rightarrow B \rightarrow D}\rightarrow_{I}$$
\end{dfn}

\begin{thm}
\textbf{Height-preserving weakening.} 
The left  rule of weakening is height-preserving admissible in ${\sf  GWFI }$.
$$ \frac{\Gamma\Rightarrow C}{A, \Gamma\Rightarrow C}L_{w}$$
\end{thm}
\begin{proof}
The proof is similar to the proof of Theorem \ref{1b}.
\end{proof}

\begin{lem}\label{23}
In ${\sf  GWFI }$ we have:
\begin{enumerate}
\item If $ \vdash_{n} A\wedge B, \Gamma \Rightarrow C $, then $ \vdash_{n} A, B, \Gamma \Rightarrow C$.
\item If $ \vdash_{n} A\vee B, \Gamma \Rightarrow C $, then $ \vdash_{n} A, \Gamma \Rightarrow C$ and $ \vdash_{n} B, \Gamma \Rightarrow C $.
\item If $ \vdash_{n}  \Gamma \Rightarrow  A\wedge B $, then $ \vdash_{n}  \Gamma \Rightarrow  A $ and $ \vdash_{n}  \Gamma \Rightarrow B $.
\end{enumerate}
\end{lem}
\begin{proof}
Similar to the proof of Lemma \ref{2}. 
\end{proof}

\begin{thm}
\textbf{Height-preserving contraction.}
The left  rule of contraction is height-preserving admissible in ${\sf  GWFI }$.
$$\frac{D, D, \Gamma\Rightarrow C}{D, \Gamma\Rightarrow C}L_{c}$$
\end{thm}
\begin{proof}
The proof is similar to the proof of Theorem \ref{1c}.
\end{proof}

First we prove the admissibility of cut for ${\sf  GWFI } $.

\begin{thm}\label{cutI}
The rule of cut  is addmissible in ${\sf  GWFI } $.
$$ \frac{\Gamma\Rightarrow  D~~~~~~D, \Gamma^{'}\Rightarrow \Delta}{\Gamma,\Gamma^{'} \Rightarrow \Delta}Cut$$
\end{thm}
\begin{proof}
By theorem \ref{1d}, we just need to consider the following cases.

If the right premiss is by rule $ \rightarrow_{I} $, we transform

$$ \frac{\genfrac{}{}{0pt}{}{}{\Gamma\Rightarrow D}~~~~\frac{D\Rightarrow B\rightarrow C~~~~D\Rightarrow C\rightarrow E}{D \Rightarrow B\rightarrow E }\rightarrow_{I}}{\Gamma \Rightarrow B\rightarrow E }Cut$$

 \begin{center}
into
\end{center}

$$\frac{\frac{\Gamma\Rightarrow D ~~~D\Rightarrow B\rightarrow C}{\Gamma\Rightarrow B\rightarrow C}Cut~~~\frac{\Gamma\Rightarrow D ~~~D\Rightarrow C\rightarrow E}{\Gamma\Rightarrow C\rightarrow E}Cut}{\Gamma \Rightarrow B\rightarrow E }\rightarrow_{I}  $$

Left premiss is by $ \rightarrow_{LR} $, and right one by $ \rightarrow_{I} $. We transform

$$\frac{\frac{F\Rightarrow  A~~~~A\Rightarrow F~~~~G\Rightarrow E~~~~E\Rightarrow G}{F\rightarrow G\Rightarrow A\rightarrow E}\rightarrow_{LR}~~~~\frac{A\rightarrow E\Rightarrow B\rightarrow C~~~~A\rightarrow E\Rightarrow C\rightarrow D}{A\rightarrow E\Rightarrow B\rightarrow D}\rightarrow_{I}}{F\rightarrow G  \Rightarrow B\rightarrow D}Cut$$
 \begin{center}
into 
\end{center}
$$ \frac{\frac{\frac{F\Rightarrow  A~~A\Rightarrow F~~G\Rightarrow E~~E\Rightarrow G}{F\rightarrow G\Rightarrow A\rightarrow E}\rightarrow_{LR}~~\genfrac{}{}{0pt}{}{}{A\rightarrow E \Rightarrow B\rightarrow C}}{F\rightarrow G  \Rightarrow B\rightarrow C}Cut~~\frac{\frac{F\Rightarrow  A~~A\Rightarrow F~~G\Rightarrow E~~E\Rightarrow G}{F\rightarrow G\Rightarrow A\rightarrow E}\rightarrow_{LR}~~\genfrac{}{}{0pt}{}{}{A\rightarrow E \Rightarrow C\rightarrow D}}{F\rightarrow G  \Rightarrow C\rightarrow D}Cut}{F\rightarrow G  \Rightarrow B\rightarrow D}\rightarrow_{I}$$

Both premises are by rule $ \rightarrow_{I} $. We transform

$$\frac{\frac{A\Rightarrow B\rightarrow C~~~~A \Rightarrow C\rightarrow D }{A\Rightarrow B\rightarrow D}\rightarrow_{I}~~~~\frac{B\rightarrow D\Rightarrow E\rightarrow G~~~~B\rightarrow D\Rightarrow G\rightarrow F }{B\rightarrow D\Rightarrow E\rightarrow F}\rightarrow_{I}}{A\Rightarrow E\rightarrow F}Cut$$

\begin{center}
into 
\end{center}
$$\frac{\frac{\frac{A\Rightarrow B\rightarrow C~~~~A \Rightarrow C\rightarrow D }{A\Rightarrow B\rightarrow D}\rightarrow_{I}~~\genfrac{}{}{0pt}{}{}{B\rightarrow D\Rightarrow E\rightarrow G}}{A\Rightarrow E\rightarrow G}Cut ~~\frac{\frac{A\Rightarrow B\rightarrow C~~~~A \Rightarrow C\rightarrow D }{A\Rightarrow B\rightarrow D}\rightarrow_{I}~~\genfrac{}{}{0pt}{}{}{B\rightarrow D \Rightarrow G\rightarrow F}}{A\Rightarrow G\rightarrow F}Cut}{A\Rightarrow E\rightarrow F}\rightarrow_{I} $$

Left premiss is by $ \rightarrow_{I} $, and right one by $ \rightarrow_{LR} $. We transform

$$\frac{\frac{A\Rightarrow  B\rightarrow C~~~~A\Rightarrow C\rightarrow D}{A\Rightarrow B\rightarrow D}\rightarrow_{I}~~~~\frac{B\Rightarrow E~~~~E\Rightarrow B~~~~D\Rightarrow F~~~~F\Rightarrow D}{B\rightarrow D\Rightarrow E\rightarrow F}\rightarrow_{LR}}{A\Rightarrow E\rightarrow F}Cut$$

\begin{center}
into 
\end{center}
$$ \frac{\frac{\frac{E\Rightarrow B}{A\Rightarrow E\rightarrow B}\rightarrow_{R}~~~\genfrac{}{}{0pt}{}{}{A\Rightarrow B\rightarrow C}}{A\Rightarrow E\rightarrow C}\rightarrow_{I}~~~\frac{\genfrac{}{}{0pt}{}{}{A\Rightarrow C\rightarrow D}~~\frac{D\Rightarrow F}{A\Rightarrow D\rightarrow F}\rightarrow_{R}}{A\Rightarrow C\rightarrow F}\rightarrow_{I}}{A\Rightarrow E\rightarrow F}\rightarrow_{I} $$

Left premiss is by $ \rightarrow_{R} $, and right one by $ \rightarrow_{I} $. We transform

$$\frac{\frac{A\Rightarrow C}{ \Gamma \Rightarrow A\rightarrow C}\rightarrow_{R}~~~~~\frac{A\rightarrow C\Rightarrow B\rightarrow E~~~~A\rightarrow C\Rightarrow E\rightarrow D}{A\rightarrow C\Rightarrow B\rightarrow D}\rightarrow_{I}}{\Gamma\Rightarrow  B\rightarrow D}Cut$$
\begin{center}
into 
\end{center}
$$\frac{\frac{\frac{A\Rightarrow C}{\Gamma \Rightarrow A\rightarrow C}\rightarrow_{R}~~\genfrac{}{}{0pt}{}{}{A\rightarrow C\Rightarrow B\rightarrow E}}{\Gamma\Rightarrow B\rightarrow E}Cut~~~\frac{\frac{A\Rightarrow C}{\Gamma \Rightarrow A\rightarrow C}\rightarrow_{R}~~\genfrac{}{}{0pt}{}{}{A\rightarrow C\Rightarrow E\rightarrow D}}{\Gamma\Rightarrow E\rightarrow D}Cut}{\Gamma\Rightarrow B\rightarrow D}\rightarrow_{I} $$

Left premiss is by $ \rightarrow_{I} $, and right one by $ \rightarrow_{R} $. The proof is easy.
\end{proof}

\begin{lem}\label{46}
If $ {\sf  GWFI  }\vdash_{n}\Rightarrow A\rightarrow B$, then $ {\sf  GWFI  }\vdash_{n} A\Rightarrow B$. 
\end{lem}
\begin{proof}
The proof is by induction on the height of the derivation of the premises.
By Lemma \ref{44}, we jus need to consider the case in which   $\Rightarrow A\rightarrow B$ is derived by $ \rightarrow_{I} $. Then it has two premises $ \Rightarrow A\rightarrow C $ and $ \Rightarrow  C\rightarrow B$. Then:
$$\frac{\frac{\Rightarrow A\rightarrow C}{A\Rightarrow  C}I\!H~~~\frac{\Rightarrow  C\rightarrow B}{C\Rightarrow  B}I\!H }{A\Rightarrow B}Cut$$
\end{proof}

Now we establish the equivalence between the Hilbert-style system {\sf WFI} and the corresponding sequent calculus, i.e. ${\sf  GWFI } $.

\begin{thm} \label{3cI}
Derivability in the sequent system ${\sf  GWFI } $ and in the Hilbert-style  system $ {\sf WFI} $ are
equivalent, i.e.
$${\sf  GWFI }  \vdash \Gamma \Rightarrow \Delta ~~~~\mbox{iff}~~~~\vdash_{\sf WFI} \bigwedge \Gamma \rightarrow \bigvee \Delta $$
\end{thm}

\begin{proof}
\textbf{Left-to-right:} By Theorem \ref{3}, we just need to assume the last derivation is by the rule $ \rightarrow_{I} $. 
 In this case we have ${\sf  GWFI }  \vdash_{n} \Gamma^{'}, \Gamma\Rightarrow B\rightarrow D, \Delta$. We need to show that  $\vdash_{\sf WFI}   \Gamma^{'} \bigwedge\Gamma \rightarrow (B\rightarrow D)\bigvee \Delta $. By assumption we have $\vdash_{n-1}\Gamma\Rightarrow  B\rightarrow C  $ and $\vdash_{n-1} \Gamma\Rightarrow C\rightarrow D $. By the induction hypothesis we conclude that  $\vdash_{\sf WFI} \Gamma \rightarrow  (B\rightarrow C) $ and $\vdash_{\sf WFI} \Gamma\rightarrow  (C\rightarrow D)  $. Using these results, we have the following proof in the Hilbert-style proof system for $  {\sf WFI}$:
 \begin{enumerate}
\item  $\vdash_{\sf WFI} \Gamma \rightarrow  (B\rightarrow C)   $~~~~~~~~~~~~~~~~by IH
\item $\vdash_{\sf WFI} \Gamma\rightarrow  (C\rightarrow D)  $~~~~~~~~~~~~~~~~by IH
\item $\vdash_{\sf WFC} \Gamma\rightarrow (B\rightarrow C)\wedge (C\rightarrow D) $~~~~~~~~~~~~by 1 and  2 
\item $ (B\rightarrow C)\wedge (C\rightarrow D)\rightarrow (B\rightarrow D)  $  ~~~~~~~~axiom {\sf C}
\item $\vdash_{\sf WFC} \Gamma \rightarrow  (B\rightarrow D)$~~~~~~~~~~~by 3, 4 and rule 9
\item $\vdash_{\sf WFC} \Gamma^{'}\bigwedge \Gamma\rightarrow \Gamma$~~~~~~~~~~~~~~ by axiom 4
\item $\vdash_{\sf WFC} \Gamma^{'}\bigwedge\Gamma\rightarrow (B\rightarrow D)$~~~~~~~~~~~~~~~~ by 5, 6
\item $\vdash_{\sf WFC}(B\rightarrow D)\rightarrow (B\rightarrow D)\bigvee\Delta  $~~~~~~~by axiom 1
\item $\vdash_{\sf WFC}\Gamma^{'}\bigwedge \Gamma \rightarrow (B\rightarrow D)\bigvee \Delta  $ ~~~~~~~~~~by 7, 8 and rule 9
\end{enumerate}

\textbf{Right-to-left:} 
By Theorem \ref{3}, we just need to show that the axiom {\sf C} can be deduced in {\sf  GWFI } as follows:

$$\frac{\frac{\frac{\frac{•}{A\rightarrow B, B\rightarrow C\Rightarrow A\rightarrow B }\ref{1N}}{(A\rightarrow B)\wedge (B\rightarrow C)\Rightarrow (A\rightarrow B)}\wedge_{L}~~\frac{\frac{•}{A\rightarrow B, B\rightarrow C\Rightarrow B\rightarrow C }\ref{1N}}{(A\rightarrow B)\wedge (B\rightarrow C)\Rightarrow (B\rightarrow C)}\wedge_{L}}{(A\rightarrow B)\wedge (B\rightarrow C)\Rightarrow (A\rightarrow C)}\rightarrow_{I}}{\Rightarrow(A\rightarrow B)\wedge (B\rightarrow C)\rightarrow (A\rightarrow C)}\rightarrow_{R} $$

\end{proof}
By combining Theorem \ref{4b} and Theorem  \ref{3cI} we have the following result.
\begin{cor}
The calculus ${\sf  GWFI }$ is sound and complete with respect to the class of  transitive $ N\!B $-neighbourhood frames.
\end{cor}

\subsection{The system ${\sf  GWFC  }$}\label{modneigh4}
If we consider $\rightarrow_{C} $ instead of $\rightarrow_{ I}$ in the system ${\sf  GWFI }$ and add the cut rule to it, then the  system ${\sf  GWFC }$ is defined.
$$\frac{\Gamma\Rightarrow B\rightarrow C~~~~\Gamma\Rightarrow B\rightarrow D}{\Gamma\Rightarrow (B\rightarrow C\wedge D)} \rightarrow_{C} $$

The proof of the admissibility of the cut rule for the system ${\sf  GWFC}$ remains an open problem.

\begin{lem}\label{47}

\

\begin{enumerate}
\item If ${\sf  GWFC } \vdash_{n}\Rightarrow A\rightarrow B$, then ${\sf  GWFC } \vdash_{n} A\Rightarrow B$. 
\item If ${\sf  GWFC }  \vdash_{n}  \Gamma \Rightarrow  A\wedge B $, then $ {\sf  GWFC } \vdash_{n}  \Gamma \Rightarrow  A $ and $ \vdash_{n}  \Gamma \Rightarrow B $.
\end{enumerate}
\end{lem}
\begin{proof}
1. The proof is by induction on the height of the derivation of the premises.
By Lemma \ref{44}, we jus need to consider the case in which   $\Rightarrow A\rightarrow B$ is derived by $ \rightarrow_{C} $ and $ B=C\wedge D $. Then it has two premises $ \Rightarrow A\rightarrow C $ and $ \Rightarrow  A\rightarrow D$. Then:
$$\frac{\frac{\Rightarrow A\rightarrow C}{A\Rightarrow  C}I\!H~~~\frac{\Rightarrow  A\rightarrow D}{A\Rightarrow  D}I\!H }{A\Rightarrow C\wedge D}\wedge_{R}$$

2. Similar to the proof of Lemma \ref{23}.
\end{proof}

In the following we establish the equivalence between the Hilbert-style system {\sf WFC} and the corresponding sequent calculus, i.e. \gwfc.

\begin{thm} \label{3cc}
Derivability in the sequent system \gwfc and in the Hilbert-style  system $ {\sf WFC} $ are
equivalent, i.e.
$$\gwfc  \vdash \Gamma \Rightarrow \Delta ~~~~\mbox{iff}~~~~\vdash_{\sf WFC} \bigwedge \Gamma \rightarrow \bigvee \Delta $$
\end{thm}

\begin{proof}
\textbf{Left-to-right:} By Theorem \ref{3}, we just need to assume the last derivation is by the rule $ \rightarrow_{C} $. 
 In this case we have ${\sf  GWFC }  \vdash_{n} \Gamma^{'}, \Gamma \Rightarrow  B\rightarrow C\wedge D, \Delta$. We need to show that  $\vdash_{\sf WFC}   \Gamma \bigwedge \Gamma^{'} \rightarrow (C\rightarrow D)\bigvee\Delta $. By assumption we have $\vdash_{n-1} \Gamma \Rightarrow  B \rightarrow C $ and $\vdash_{n-1} \Gamma\Rightarrow B\rightarrow D   $. By the induction hypothesis we conclude that  $\vdash_{\sf WFC} \Gamma \rightarrow  (B \rightarrow C  )$ and $\vdash_{\sf WFC} \Gamma \rightarrow ( B \rightarrow D ) $. Using these results, we have the following proof in the Hilbert-style proof system for $  {\sf WFC}$:
 \begin{enumerate}
\item  $\vdash_{\sf WFC}\Gamma \rightarrow  (B \rightarrow C  ) $~~~~~~~~~~~~~~~~by IH
\item $\vdash_{\sf WFC} \Gamma \rightarrow ( B \rightarrow D ) $~~~~~~~~~~~~~~~~by IH
\item $\vdash_{\sf WFC}\Gamma \rightarrow (B \rightarrow C  ) \wedge ( B \rightarrow D ) $~~~~~~~~~~~~~~~~by 1 and 2
\item $\vdash_{\sf WFC}  (B \rightarrow C  ) \wedge ( B \rightarrow D ) \rightarrow  ( B \rightarrow C\wedge D ) $~~~~~~~ by axiom {\sf C}
\item $\vdash_{\sf WFC}  \Gamma \rightarrow  ( B \rightarrow C\wedge D ) $ ~~~~~~~~~by 3, 4
\item $\vdash_{\sf WFC} \Gamma \bigwedge\Gamma^{'} \rightarrow (C\rightarrow D)\bigvee \Delta $
\end{enumerate}

 \textbf{Right-to-left:} By Theorem \ref{3}, we just need to show that the axiom {\sf C} can be deduced in {\sf GWFC} as follows:
 
  $$\frac{\frac{\frac{\frac{•}{A\rightarrow B, A\rightarrow C\Rightarrow A\rightarrow B }\ref{1N}}{(A\rightarrow B)\wedge (A\rightarrow C)\Rightarrow (A\rightarrow B)}\wedge_{L}~~\frac{\frac{•}{A\rightarrow B, A\rightarrow C\Rightarrow A\rightarrow C }\ref{1N}}{(A\rightarrow B)\wedge (A\rightarrow C)\Rightarrow (A\rightarrow C)}\wedge_{L}}{(A\rightarrow B)\wedge (A\rightarrow C)\Rightarrow  (A\rightarrow B\wedge C)}\rightarrow_{C}}{\Rightarrow (A\rightarrow B)\wedge (A\rightarrow C)\rightarrow (A\rightarrow B\wedge C)}\rightarrow_{R} $$
  
\end{proof}

By combining Theorem \ref{4b} and Theorem  \ref{3cc} we have the following result.
\begin{cor}
The calculus ${\sf  GWFC }$ is sound and complete with respect to the class of $ N\!B $-neighbourhood frames that are closed under intersection.
\end{cor}

\subsection{The system ${\sf  GWFD } $}\label{modneigh5}
If we consider $\rightarrow_{D} $ instead of $\rightarrow_{ I}$ in the system ${\sf  GWFI }$ and add the cut rule to it, then the  system ${\sf  GWFD}$ is defined.
$$\frac{\Gamma\Rightarrow B\rightarrow C~~~~\Gamma\Rightarrow D\rightarrow C}{\Gamma \Rightarrow (B\vee D\rightarrow C)} \rightarrow_{D} $$
The proof of the admissibility of the cut rule for the system ${\sf  GWFD}$ remains an open problem, and I have not yet been able to provide a proof.

\begin{lem}\label{48}

\
\begin{enumerate}
\item If ${\sf  GWFD } \vdash_{n}\Rightarrow A\rightarrow B$, then ${\sf  GWFD } \vdash_{n} A\Rightarrow B$. 
\item If ${\sf  GWFD }  \vdash_{n}  \Gamma \Rightarrow  A\wedge B $, then $ {\sf  GWFD } \vdash_{n}  \Gamma \Rightarrow  A $ and $ \vdash_{n}  \Gamma \Rightarrow B $.
\end{enumerate}
\end{lem}
\begin{proof}
1. The proof is by induction on the height of the derivation of the premises.
By Lemma \ref{44}, we jus need to consider the case in which   $\Rightarrow A\rightarrow B$ is derived by $ \rightarrow_{D} $ and $ A=C\vee D $. Then it has two premises $ \Rightarrow C\rightarrow B $ and $ \Rightarrow  D\rightarrow B$. Then:
$$\frac{\frac{\Rightarrow C\rightarrow B}{C\Rightarrow  B}I\!H~~~\frac{\Rightarrow  D\rightarrow B}{D\Rightarrow  B}I\!H }{C\vee D\Rightarrow B}\vee_{L}$$

2. Similar to the proof of Lemma \ref{23}.
\end{proof}

Now we establish the equivalence between the Hilbert-style system {\sf WFD} and the corresponding sequent calculus, i.e. ${\sf  GWFD } $.

\begin{thm} \label{3cdcd}
Derivability in the sequent system ${\sf  GWFD } $ and in the Hilbert-style  system $ {\sf WFD} $ are
equivalent, i.e.
$${\sf  GWFD }  \vdash \Gamma \Rightarrow \Delta ~~~~\mbox{iff}~~~~\vdash_{\sf WFD} \bigwedge \Gamma \rightarrow \bigvee \Delta $$
\end{thm}

\begin{proof}
The proof is similar to the proof of Theorem \ref{3cc}.
\end{proof}

By combining Theorem \ref{4b} and Theorem  \ref{3cdcd} we have the following result.
\begin{cor}
The calculus ${\sf  GWFD}$ is sound and complete with respect to the class of $ N\!B $-neighbourhood frames that are closed under union.
\end{cor}

\subsection{The system {\sf GF}}
The system ${\sf  GF }$ is defined as $\gwf \oplus \rightarrow_{\sf D}\oplus \rightarrow_{\sf C}$.

\begin{thm}
\textbf{Height-preserving weakening.} 
The left  rule of weakening is height-preserving admissible in ${\sf  GF }$.
$$ \frac{\Gamma\Rightarrow C}{A, \Gamma\Rightarrow C}L_{w}$$
\end{thm}
\begin{proof}
The proof is similar to the proof of Theorem \ref{1b}.
\end{proof}

\begin{lem}\label{23a}
In ${\sf  GF}$ we have:
\begin{enumerate}
\item If $ \vdash_{n} A\wedge B, \Gamma \Rightarrow C $, then $ \vdash_{n} A, B, \Gamma \Rightarrow C$.
\item If $ \vdash_{n} A\vee B, \Gamma \Rightarrow C $, then $ \vdash_{n} A, \Gamma \Rightarrow C$ and $ \vdash_{n} B, \Gamma \Rightarrow C $.
\item If $ \vdash_{n}  \Gamma \Rightarrow  A\wedge B $, then $ \vdash_{n}  \Gamma \Rightarrow   A $ and $ \vdash_{n}  \Gamma \Rightarrow   B $.
\end{enumerate}
\end{lem}
\begin{proof}
Similar to the proof of Lemma \ref{2}. 
\end{proof}

\begin{thm}
\textbf{Height-preserving contraction.}
The left  rule of contraction is height-preserving admissible in ${\sf  GF}$.
$$\frac{D, D, \Gamma\Rightarrow C}{D, \Gamma\Rightarrow C}L_{c}$$
\end{thm}
\begin{proof}
The proof is similar to the proof of Theorem \ref{1c}.
\end{proof}

First we prove the admissibility of cut for ${\sf  GF } $.

\begin{thm}\label{cutF}
The rule of cut  is addmissible in ${\sf  GF } $.
$$ \frac{\Gamma\Rightarrow  D~~~~~~D, \Gamma^{'}\Rightarrow \Delta}{\Gamma,\Gamma^{'} \Rightarrow \Delta}Cut$$
\end{thm}
\begin{proof}
By theorem \ref{cutI}, we just need to consider the following cases.

If the right premiss is by rule $ \rightarrow_{C} $, we transform

$$ \frac{\genfrac{}{}{0pt}{}{}{\Gamma\Rightarrow D}~~~~\frac{D\Rightarrow B\rightarrow C~~~~D\Rightarrow B\rightarrow D}{D \Rightarrow B\rightarrow C\wedge D }\rightarrow_{C}}{\Gamma \Rightarrow B\rightarrow C\wedge D }Cut$$

 \begin{center}
into
\end{center}

$$\frac{\frac{\Gamma\Rightarrow D ~~~D\Rightarrow B\rightarrow C}{\Gamma\Rightarrow B\rightarrow C}Cut~~~\frac{\Gamma\Rightarrow D ~~~D\Rightarrow B\rightarrow D}{\Gamma\Rightarrow B\rightarrow D}Cut}{\Gamma \Rightarrow B\rightarrow C\wedge D }\rightarrow_{C}  $$

If the right premiss is by rule $ \rightarrow_{D} $, we transform as above.

Left premiss is by $ \rightarrow_{LR} $, and right one by $ \rightarrow_{C} $. We transform

$$\frac{\frac{F\Rightarrow  A~~~~A\Rightarrow F~~~~G\Rightarrow E~~~~E\Rightarrow G}{F\rightarrow G\Rightarrow A\rightarrow E}\rightarrow_{LR}~~~~\frac{A\rightarrow E\Rightarrow B\rightarrow C~~~~A\rightarrow E\Rightarrow B\rightarrow D}{A\rightarrow E\Rightarrow B\rightarrow C\wedge D}\rightarrow_{C}}{F\rightarrow G  \Rightarrow B\rightarrow C\wedge D}Cut$$
 \begin{center}
into 
\end{center}
$$ \frac{\frac{\frac{F\Rightarrow  A~~A\Rightarrow F~~G\Rightarrow E~~E\Rightarrow G}{F\rightarrow G\Rightarrow A\rightarrow E}\rightarrow_{LR}~~\genfrac{}{}{0pt}{}{}{A\rightarrow E \Rightarrow B\rightarrow C}}{F\rightarrow G  \Rightarrow B\rightarrow C}Cut~~\frac{\frac{F\Rightarrow  A~~A\Rightarrow F~~G\Rightarrow E~~E\Rightarrow G}{F\rightarrow G\Rightarrow A\rightarrow E}\rightarrow_{LR}~~\genfrac{}{}{0pt}{}{}{A\rightarrow E \Rightarrow B\rightarrow D}}{F\rightarrow G  \Rightarrow B\rightarrow D}Cut}{F\rightarrow G  \Rightarrow B\rightarrow C\wedge D}\rightarrow_{C}$$

Left premiss is by $ \rightarrow_{LR} $, and right one by $ \rightarrow_{D} $. Same as above.

Both premises are by rule $ \rightarrow_{C} $. We transform

$$\frac{\frac{\Gamma\Rightarrow A\rightarrow B~~~~\Gamma \Rightarrow A\rightarrow C}{\Gamma\Rightarrow A\rightarrow B\wedge C}\rightarrow_{C}~~~~\frac{A\rightarrow B\wedge C\Rightarrow E\rightarrow F~~~~A\rightarrow B\wedge C\Rightarrow E\rightarrow D }{A\rightarrow B\wedge C\Rightarrow E\rightarrow F\wedge D}\rightarrow_{C}}{\Gamma\Rightarrow E\rightarrow F\wedge D}Cut$$

\begin{center}
into 
\end{center}
$$\frac{\frac{\frac{\Gamma\Rightarrow A\rightarrow B~~~~\Gamma \Rightarrow A\rightarrow C}{\Gamma\Rightarrow A\rightarrow B\wedge C}\rightarrow_{C}~~\genfrac{}{}{0pt}{}{}{A\rightarrow B\wedge C\Rightarrow E\rightarrow F}}{\Gamma\Rightarrow E\rightarrow F}Cut ~~\frac{\frac{\Gamma\Rightarrow A\rightarrow B~~~~\Gamma \Rightarrow A\rightarrow C}{\Gamma\Rightarrow A\rightarrow B\wedge C}\rightarrow_{C}~~\genfrac{}{}{0pt}{}{}{A\rightarrow B\wedge C \Rightarrow E\rightarrow D}}{\Gamma\Rightarrow E\rightarrow D}Cut}{\Gamma\Rightarrow E\rightarrow F\wedge D}\rightarrow_{I} $$

Both premises are by rule $ \rightarrow_{D} $. Same as above.

Left premiss is by $ \rightarrow_{D} $, and right one by $ \rightarrow_{LR} $. We transform

$$\frac{\frac{\Gamma\Rightarrow  A\rightarrow B~~~~\Gamma\Rightarrow C\rightarrow B}{\Gamma\Rightarrow A\vee C\rightarrow B}\rightarrow_{D}~~~~\frac{A\vee C\Rightarrow E~~~~E\Rightarrow A\vee C~~~~B\Rightarrow F~~~~F\Rightarrow B}{A\vee C\rightarrow B \Rightarrow E\rightarrow F}\rightarrow_{LR}}{\Gamma\Rightarrow E\rightarrow F}Cut$$

\begin{center}
into 
\end{center}
$$ \frac{\frac{\frac{E\Rightarrow A\vee C}{\Gamma\Rightarrow E\rightarrow A\vee C}\rightarrow_{R}~~~\frac{\Gamma\Rightarrow  A\rightarrow B~~~~\Gamma\Rightarrow C\rightarrow B}{\Gamma\Rightarrow A\vee C\rightarrow B}\rightarrow_{D}}{\Gamma\Rightarrow E\rightarrow B}\rightarrow_{I}~~~\frac{B\Rightarrow F}{\Gamma\Rightarrow B\rightarrow F}\rightarrow_{R}}{\Gamma\Rightarrow E\rightarrow F}\rightarrow_{I} $$

Left premiss is by $ \rightarrow_{C} $, and right one by $ \rightarrow_{LR} $. We transform

$$\frac{\frac{\Gamma\Rightarrow  A\rightarrow B~~~~\Gamma\Rightarrow A\rightarrow C}{\Gamma\Rightarrow A \rightarrow B\wedge C}\rightarrow_{C}~~~~\frac{A\Rightarrow E~~~~E\Rightarrow A~~~~B\wedge C\Rightarrow F~~~~F\Rightarrow B\wedge C}{A \rightarrow B\wedge C \Rightarrow E\rightarrow F}\rightarrow_{LR}}{\Gamma\Rightarrow E\rightarrow F}Cut$$

\begin{center}
into 
\end{center}
$$ \frac{\frac{E\Rightarrow A}{\Gamma\Rightarrow E\rightarrow A}\rightarrow_{R}~~~~\frac{\frac{\Gamma\Rightarrow  A\rightarrow B~~~~\Gamma\Rightarrow A\rightarrow C}{\Gamma\Rightarrow A \rightarrow B\wedge C}\rightarrow_{C}~~~\frac{B\wedge C\Rightarrow F}{\Gamma\Rightarrow B\wedge C\rightarrow F}\rightarrow_{R}}{\Gamma\Rightarrow A\rightarrow F}\rightarrow_{I}}{\Gamma\Rightarrow E\rightarrow F}\rightarrow_{I} $$

Left premiss is by $ \rightarrow_{I} $, and right one by $ \rightarrow_{C} $. We transform

$$\frac{\frac{\Gamma\Rightarrow A\rightarrow B~~~~\Gamma \Rightarrow B\rightarrow C}{\Gamma\Rightarrow A\rightarrow C}\rightarrow_{I}~~~~\frac{A\rightarrow  C\Rightarrow D\rightarrow E~~~~A\rightarrow  C\Rightarrow D\rightarrow F }{A\rightarrow   C\Rightarrow D\rightarrow E\wedge F}\rightarrow_{C}}{\Gamma\Rightarrow D\rightarrow E\wedge F}Cut$$
\begin{center}
into 
\end{center}
$$ \frac{\frac{\frac{\Gamma\Rightarrow A\rightarrow B~~~~\Gamma \Rightarrow B\rightarrow C}{\Gamma\Rightarrow A\rightarrow C}\rightarrow_{I}~~~\genfrac{}{}{0pt}{}{}{A\rightarrow  C\Rightarrow D\rightarrow E}}{\Gamma\Rightarrow D\rightarrow E}Cut ~~~\frac{\frac{\Gamma\Rightarrow A\rightarrow B~~~~\Gamma \Rightarrow B\rightarrow C}{\Gamma\Rightarrow A\rightarrow C}\rightarrow_{I}~~~\genfrac{}{}{0pt}{}{}{A\rightarrow  C\Rightarrow D\rightarrow F}}{\Gamma\Rightarrow D\rightarrow  F}Cut}{\Gamma\Rightarrow D\rightarrow E\wedge F} \rightarrow_{C}$$

Left premiss is by $ \rightarrow_{I} $, and right one by $ \rightarrow_{D} $. Same as above.

Left premiss is by $ \rightarrow_{C} $, and right one by $ \rightarrow_{I} $. We transform

$$\frac{\frac{\Gamma\Rightarrow A\rightarrow B~~~~\Gamma \Rightarrow A\rightarrow C}{\Gamma\Rightarrow A\rightarrow B\wedge C}\rightarrow_{C}~~~~\frac{A\rightarrow  B\wedge C\Rightarrow E\rightarrow C~~~~A\rightarrow  B\wedge C\Rightarrow C\rightarrow F }{A\rightarrow   B\wedge C\Rightarrow E\rightarrow  F}\rightarrow_{I}}{\Gamma\Rightarrow E\rightarrow  F}Cut$$
\begin{center}
into 
\end{center}
$$ \frac{\frac{\frac{\Gamma\Rightarrow A\rightarrow B~~~~\Gamma \Rightarrow A\rightarrow C}{\Gamma\Rightarrow A\rightarrow B\wedge C}\rightarrow_{C}~~~\genfrac{}{}{0pt}{}{}{A\rightarrow B\wedge C\Rightarrow E\rightarrow C}}{\Gamma\Rightarrow E\rightarrow C}Cut ~~~\frac{\frac{\Gamma\Rightarrow A\rightarrow B~~~~\Gamma \Rightarrow A\rightarrow C}{\Gamma\Rightarrow A\rightarrow B\wedge C}\rightarrow_{C}~~~\genfrac{}{}{0pt}{}{}{A\rightarrow  B\wedge C\Rightarrow C\rightarrow F}}{\Gamma\Rightarrow C\rightarrow  F}Cut}{\Gamma\Rightarrow E\rightarrow  F} \rightarrow_{I}$$

Left premiss is by $ \rightarrow_{D} $, and right one by $ \rightarrow_{I} $. Same as above.

Left premiss is by $ \rightarrow_{C} $, and right one by $ \rightarrow_{D} $. Same as above.

Left premiss is by $ \rightarrow_{D} $, and right one by $ \rightarrow_{C} $. Same as above.

Left premiss is by $ \rightarrow_{R} $, and right one by $ \rightarrow_{C} $. We transform
$$\frac{\frac{A\Rightarrow C}{ \Gamma \Rightarrow A\rightarrow C}\rightarrow_{R}~~~~~\frac{A\rightarrow C\Rightarrow B\rightarrow D~~~~A\rightarrow C\Rightarrow B\rightarrow E}{A\rightarrow C\Rightarrow B\rightarrow D\wedge E}\rightarrow_{C}}{\Gamma\Rightarrow  B\rightarrow D\wedge E}Cut$$
\begin{center}
into 
\end{center}
$$\frac{\frac{\frac{A\Rightarrow C}{\Gamma \Rightarrow A\rightarrow C}\rightarrow_{R}~~\genfrac{}{}{0pt}{}{}{A\rightarrow C\Rightarrow B\rightarrow D}}{\Gamma\Rightarrow B\rightarrow D}Cut~~~\frac{\frac{A\Rightarrow C}{\Gamma \Rightarrow A\rightarrow C}\rightarrow_{R}~~\genfrac{}{}{0pt}{}{}{A\rightarrow C\Rightarrow B\rightarrow E}}{\Gamma\Rightarrow B\rightarrow E}Cut}{\Gamma\Rightarrow B\rightarrow D\wedge E}\rightarrow_{I} $$

Left premiss is by $ \rightarrow_{R} $, and right one by $ \rightarrow_{D} $. Same as above.

Left premiss is by $ \rightarrow_{D} $, and right one by $ \rightarrow_{R} $. This case is obvious.

Left premiss is by $ \rightarrow_{C} $, and right one by $ \rightarrow_{R} $. This case is obvious.
\end{proof}
\begin{lem}\label{51}
If ${\sf  GF } \vdash_{n}\Rightarrow A\rightarrow B$, then $ {\sf  GF }\vdash_{n} A\Rightarrow B$. 
\end{lem}
\begin{proof}
By Lemmas \ref{46}, \ref{47} and \ref{48}.
\end{proof}

Now we establish the equivalence between the Hilbert-style system {\sf F} and the corresponding sequent calculus, i.e. ${\sf  GF} $.

\begin{thm} \label{ffff}
Derivability in the sequent system ${\sf  GF } $ and in the Hilbert-style  system $ {\sf F} $ are
equivalent, i.e.
$${\sf  GF }  \vdash \Gamma \Rightarrow \Delta ~~~~\mbox{iff}~~~~\vdash_{\sf F} \bigwedge \Gamma \rightarrow \bigvee \Delta $$
\end{thm}

\begin{proof}
We know that the logic \f is the smallest set of formulas closed under instances of \wf, {\sf C}, {\sf D} and {\sf I} \cite{FD}. Hence, the proof of this theorem is obvious by Theorems \ref{3cI}, \ref{3cc} and \ref{3cdcd}.
\end{proof}

By combining Theorem \ref{4b} and Theorem  \ref{ffff} we have the following result.
\begin{cor}
The calculus ${\sf  GF }$ is sound and complete with respect to the class of  transitive $ N\!B $-neighbourhood frames that are closed under conjunction and disjunction.
\end{cor}

\subsection{The system {\sf GWFCI}}
The system {\sf GWFCI} is defined as ${\sf GWFI}\oplus  \rightarrow_{C}$.
\begin{thm}\label{cutGWFCI}
The rule of cut  is addmissible in ${\sf  GWFCI } $.
$$ \frac{\Gamma\Rightarrow  D~~~~~~D, \Gamma^{'}\Rightarrow \Delta}{\Gamma,\Gamma^{'} \Rightarrow \Delta}Cut$$
\end{thm}
\begin{proof}
Refer to the proof of Theorem \ref{cutF}.
\end{proof}

\begin{thm} \label{GWFCI}
Derivability in the sequent system ${\sf  GWFCI  } $ and in the Hilbert-style  system $ {\sf WFCI} $ are
equivalent, i.e.
$${\sf  GWFCI }  \vdash \Gamma \Rightarrow \Delta ~~~~\mbox{iff}~~~~\vdash_{\sf WFCI} \bigwedge \Gamma \rightarrow \bigvee \Delta $$
\end{thm}

\begin{proof}
By Theorems \ref{ffff}.
\end{proof}

By combining Theorem \ref{4b} and Theorem  \ref{GWFCI} we have the following result.
\begin{cor}
The calculus ${\sf  GWFCI  }$ is sound and complete with respect to the class of  transitive $ N\!B $-neighbourhood frames that are closed under conjunction.
\end{cor}

\subsection{The system {\sf GWFDI}}
The system {\sf GWFDI} is defined as ${\sf GWFI}\oplus  \rightarrow_{D}$.

\begin{thm}\label{cutGWFDI}
The rule of cut  is addmissible in ${\sf  GWFDI } $.
$$ \frac{\Gamma\Rightarrow  D~~~~~~D, \Gamma^{'}\Rightarrow \Delta}{\Gamma,\Gamma^{'} \Rightarrow \Delta}Cut$$
\end{thm}
\begin{proof}
Refer to the proof of Theorem \ref{cutF}.
\end{proof}
\begin{thm} \label{GWFDI }
Derivability in the sequent system ${\sf  GWFDI  } $ and in the Hilbert-style  system $ {\sf WFDI} $ are
equivalent, i.e.
$${\sf  GWFDI }  \vdash \Gamma \Rightarrow \Delta ~~~~\mbox{iff}~~~~\vdash_{\sf WFDI} \bigwedge \Gamma \rightarrow \bigvee \Delta $$
\end{thm}

\begin{proof}
By Theorems \ref{ffff}.
\end{proof}

By combining Theorem \ref{4b} and Theorem  \ref{GWFDI } we have the following result.
\begin{cor}
The calculus ${\sf  GWFDI  }$ is sound and complete with respect to the class of  transitive $ N\!B $-neighbourhood frames that are closed under union.
\end{cor}

\section{Conclusion}
 
In this article, we presented the first analytic proof systems for a number of weak subintuitionistic logics.
We have demonstrated that the most significant systems of these calculi have good structural properties in that weakening and
contraction are height-preserving admissible and cut is syntactically admissible.
Given these findings, it seems worthwhile to further investigate interpolation for these subintuitionistic logics.
Another possible direction for a future work, particularly for the systems where we have been unable to prove cut admissibility, would be to explore labelled sequent calculi for subintuitionistic logics, which may provide a solution to the cut admissibility problem.

\vspace{1cm}

\noindent\textbf{Acknowledgements.} The author thanks Dick de Jongh, Marianna Girlando and Raheleh Jalali for all their helpful suggestions concerning subintuitionistic logics and sequent calculus. 

\label{references}
\

\end{document}